\documentclass[11pt]{amsart} \textwidth=14.5cm \oddsidemargin=1cm
\usepackage{amsmath,wasysym}
\usepackage{amsxtra}
\usepackage{amscd}
\usepackage{amsthm}
\usepackage{amsfonts}
\usepackage{amssymb}
\usepackage{eucal}
\usepackage{graphics, color}
\usepackage{hyperref}

\input prepictex
\input pictex
\input postpictex
\newtheorem{thm}{Theorem}[section]
\newtheorem{lem}[thm]{Lemma}
\newtheorem{cor}[thm]{Corollary}
\newtheorem{prop}[thm]{Proposition}

\theoremstyle{definition}

\newtheorem{example}[thm]{Example}

\theoremstyle{remark}
\newtheorem{rem}[thm]{Remark}

\numberwithin{equation}{section}

\begin{document}

%Referring commands:
\newcommand{\thmref}[1]{Theorem~\ref{#1}}
\newcommand{\secref}[1]{Section~\ref{#1}}
\newcommand{\lemref}[1]{Lemma~\ref{#1}}
\newcommand{\propref}[1]{Proposition~\ref{#1}}
\newcommand{\corref}[1]{Corollary~\ref{#1}}
\newcommand{\remref}[1]{Remark~\ref{#1}}
\newcommand{\eqnref}[1]{(\ref{#1})}
\newcommand{\exref}[1]{Example~\ref{#1}}

%Simplified symbols:
\newcommand{\nc}{\newcommand}
\nc{\Z}{{\mathbb Z}}
\nc{\C}{{\mathbb C}}
\nc{\N}{{\mathbb N}}
\nc{\F}{{\mf F}}
\nc{\Q}{\ol{Q}}
\nc{\la}{\lambda}
\nc{\ep}{\epsilon}
\nc{\h}{\mathfrak h}
\nc{\n}{\mf n}
\nc{\A}{{\mf a}}
\nc{\G}{{\mathfrak g}}
\nc{\SG}{\overline{\mathfrak g}}
\nc{\DG}{\widetilde{\mathfrak g}}
\nc{\D}{\mc D} \nc{\Li}{{\mc L}} \nc{\La}{\Lambda} \nc{\is}{{\mathbf
i}} \nc{\V}{\mf V} \nc{\bi}{\bibitem} \nc{\NS}{\mf N}
\nc{\dt}{\mathord{\hbox{${\frac{d}{d t}}$}}} \nc{\E}{\mc E}
\nc{\ba}{\tilde{\pa}} \nc{\half}{\frac{1}{2}} \nc{\mc}{\mathcal}
\nc{\mf}{\mathfrak} \nc{\hf}{\frac{1}{2}}
\nc{\hgl}{\widehat{\mathfrak{gl}}} \nc{\gl}{{\mathfrak{gl}}}
\nc{\hz}{\hf+\Z}
\nc{\dinfty}{{\infty\vert\infty}} \nc{\SLa}{\overline{\Lambda}}
\nc{\SF}{\overline{\mathfrak F}} \nc{\SP}{\overline{\mathcal P}}
\nc{\U}{\mathfrak u} \nc{\SU}{\overline{\mathfrak u}}
\nc{\ov}{\overline}
\nc{\wt}{\widetilde}
\nc{\wh}{\widehat}
\nc{\sL}{\ov{\mf{l}}}
\nc{\sP}{\ov{\mf{p}}}
\nc{\osp}{\mf{osp}}
\nc{\spo}{\mf{spo}}
\nc{\hosp}{\widehat{\mf{osp}}}
\nc{\hspo}{\widehat{\mf{spo}}}
\nc{\I}{\mathbb{I}}
\nc{\X}{\mathbb{X}}
\nc{\hh}{\widehat{\mf{h}}}
\nc{\Icirc}{I_{\text{\,\begin{picture}(2,2)\setlength{\unitlength}{0.07in} \put(0.3,0.45){\circle{1}}\end{picture}}}\,}
\nc{\Ibullet}{I_{\text{\,\begin{picture}(2,2)\setlength{\unitlength}{0.07in} \put(0.3,0.45){\circle*{1}}\end{picture}}}\,}

\newcommand{\blue}[1]{{\color{blue}#1}}
\newcommand{\red}[1]{{\color{red}#1}}
\newcommand{\green}[1]{{\color{green}#1}}
\newcommand{\white}[1]{{\color{white}#1}}

 \advance\headheight by 2pt

\title
{Irreducible Characters of Kac-Moody Lie superalgebras}

\author[Cheng]{Shun-Jen Cheng$^\dagger$}
\thanks{$^\dagger$Partially supported by an NSC grant}
\address{Institute of Mathematics, Academia Sinica, Taipei,
Taiwan 10617} \email{chengsj@math.sinica.edu.tw}

\author[Kwon]{Jae-Hoon Kwon$^{\dagger\dagger}$}
\thanks{$^{\dagger\dagger}$Partially supported by a NRF-grant 2011-0006735.}
\address{Department of Mathematics, Sungkyunkwan University
2066 Seobu-ro, Jangan-gu, Suwon, Korea}
\email{jaehoonkw@skku.edu}

\author[Wang]{Weiqiang Wang$^{\dagger\dagger\dagger}$}
\thanks{$^{\dagger\dagger\dagger}$Partially supported by an NSF grant.}
\address{Department of Mathematics, University of Virginia, Charlottesville, VA 22904}
\email{ww9c@virginia.edu}

\begin{abstract}
Generalizing the super duality formalism for finite-dimensional Lie
superalgebras of type $ABCD$, we establish an equivalence between
parabolic BGG categories of a Kac-Moody Lie superalgebra and a
Kac-Moody Lie algebra. The characters %and Kostant $\mf u$-homology groups
for a large family of irreducible highest weight modules over
a symmetrizable Kac-Moody Lie superalgebra are then given in terms
of Kazhdan-Lusztig polynomials for the first time. We formulate a notion of integrable
modules over a symmetrizable Kac-Moody Lie superalgebra via super
duality, and show  that these integrable modules form a semisimple
tensor subcategory, whose Littlewood-Richardson tensor product
multiplicities coincide with those in the Kac-Moody Lie algebra
setting.
\end{abstract}

\subjclass[2010]{17B67}

\maketitle

  \setcounter{tocdepth}{1}
 \tableofcontents

\section{Introduction}

Super duality is a powerful general approach developed in the past
years in representation theory of finite-dimensional classical Lie
superalgebras \cite{CWZ, CW1, CL, CLW} (see the book
\cite[Chapter~6]{CW} for an exposition). It states that certain
parabolic BGG categories of Lie superalgebras of type $ABCD$ and
classical Lie algebras at infinite-rank limit are equivalent as
highest weight categories. A parabolic BGG category of Lie
superalgebras almost always contains infinite-dimensional simple
modules, and at present super duality appears to be the only known
general approach toward the basic problem of finding irreducible
characters for such categories. Some notable consequences of super
duality include:

\begin{itemize}
\item[(i)]
The Kostant $\mf u$-homology groups with coefficients in irreducible
highest weight modules over classical Lie algebras and Lie
superalgebras match perfectly.

\item[(ii)]
The character formula for a large class of irreducible highest
weight modules over Lie superalgebras of type $ABCD$ is obtained via
the classical Kazhdan-Lusztig polynomials.
\end{itemize}

The goal of this paper is to {establish the super duality
formalism and also formulate a notion of integrable modules for} a
large class of Kac-Moody Lie superalgebras. The class of Kac-Moody
Lie superalgebras considered in this paper is a family of
contragredient Lie superalgebras  $\SG_n$ for $n\in \N\cup\{-1\}$
associated to super generalized Cartan matrices
corresponding to Dynkin diagrams of the form \eqref{SG diagram}, whose submatrix
corresponding to  a head subdiagram  \eqref{Head diagram}
satisfies the mild conditions \eqref{eq:pseudoCartan for hd} and
\eqref{eq:locally nilpotent} (see Sections \ref{Three Lie superalgebras} and \ref{Irreducible characters}).
More generally, we may also allow the case
where there is more than one isotropic odd simple root together with its associated tail
diagram of type $A$ attached to \framebox{$\mf{B}$}\ (see Remark~\ref{more than one tail}).
In particular,  this contains the finite-dimensional
Lie superalgebras of type $ABCD$, the exceptional simple Lie
superalgebras, and also the affine Lie superalgebras associated with classical and exceptional Lie superalgebras.

The notion of Kac-Moody (or rather that of contragredient) Lie
superalgebras was introduced by Kac \cite[Section 2.5]{K1} in a way
similar to the more familiar notion of Kac-Moody Lie algebras
(see also \cite{vdL}).
The anisotropic Kac-Moody superalgebras and their integrable representations were studied in depth
in \cite{K2} (see also \cite{CFLW} for more general representations),
where the
``anisotropic" condition (C$1'$) in Section~ \ref{Basics on g(A)} means ``no isotropic
odd simple roots".
However, little is known about
the characters of irreducible highest weight modules in BGG
categories of the general Kac-Moody Lie superalgebras (with
isotropic odd simple roots) beyond the finite-dimensional classical
Lie superalgebras except in some rather special cases for affine Lie
superalgebras (see Kac-Wakimoto \cite{KW2}). The classification
problem of finite-growth contragredient Lie superalgebras under
various conditions was studied by van der Leur \cite{vdL}, Hoyt, and
Serganova \cite{HS, Hoyt}.

For our purpose, we introduce another family of contragredient Lie
superalgebras $\G_n$ for $n\in \N$, which are obtained by replacing
all the odd isotropic simple roots in $\SG_n$ with even
non-isotropic ones. We show as a main result in this paper that
suitable parabolic BGG categories of $\SG_n$ and $\G_n$ at
infinite-rank limit are equivalent as highest weight categories. We
remark that $\G_n$ is not necessarily of Kac-Moody (super) type
since the associated matrix may not be a super generalized Cartan
matrix, as it may have a positive off-diagonal entry in a row
corresponding to non-isotropic simple root.

In the most important cases when $\G_n$  or its limit $\G$ of
infinite rank is a symmetrizable Kac-Moody Lie algebra (or even
symmetrizable anisotropic Kac-Moody superalgebra), feature (i)
above of the super duality formalism remains valid in the current
setting once we replace the word ``classical" above by ``Kac-Moody",
while (ii) also applies in many cases by making use of the solution
of the Kazhdan-Lusztig conjectures for symmetrizable Kac-Moody Lie
algebras by Kashiwara-Tanisaki and others (see \cite{KaTa2} and
references therein). Feature (ii) above applies to the simple
exceptional Lie superalgebras, but not to the affine Lie
superalgebras since the corresponding Lie algebra $\G_n$ is not of
Kac-Moody type.
%Nevertheless, as a remarkable consequence, the irreducible characters
%of such a Kac-Moody Lie superalgebra $\SG_n$ can be obtained
%from those of contragredient Lie superalgebras $\G$ without odd isotropic
%simple roots by applying the standard involution on the ring of symmetric functions.

While super duality in our general setting is established largely by
the same strategy as in the case of finite-dimensional classical Lie
superalgebras treated earlier, the technical details are more
challenging than before. For instance, the structures of root
systems are very explicit and well understood in the classical
setting, but it is much less so for Kac-Moody or contragredient Lie
superalgebras. We need to deal with issues like when a contragredient Lie
superalgebra associated to a matrix $B$ is a subalgebra of another
contragredient Lie superalgebra associated to a matrix $A$ (see
Propositions~\ref{prop:KM:embed} and \ref{prop:GSG:sub}). One also
needs to be more roundabout, without referring to the detailed
knowledge of root systems, in showing that the super duality
functors match the parabolic Verma and simple modules, respectively
(see Lemma~\ref{lem:DeltaL}). In the classical setting, we have
weights $\epsilon_i$ available which greatly facilitate the
construction of super duality. In the current generality, we have
found a way of introducing such $\epsilon_i$ as needed, by
considering an extension by an outer derivation (see \eqref{epsilon and h}).

As another remarkable application of super duality, we formulate
a notion of integrable modules over  symmetrizable Kac-Moody Lie
superalgebras $\mc{G}$ associated with super generalized Cartan
matrices with no positive off-diagonal entry (see the condition \eqref{eq:neq0} in Section \ref{sec:integrable}).
These integrable modules correspond to the usual integrable modules over
a symmetrizable anisotropic Kac-Moody Lie superalgebra $\G$  (see
\cite{K2,K}), under super duality with $\SG_{-1}=\mc{G}$. In fact, they are
shown to form a  semisimple tensor subcategory, whose
Littlewood-Richardson tensor product multiplicities coincide with
those for $\G$ (compare \cite{Kw}). We show that the integrable
$\mc{G}$-modules afford a nice intrinsic characterization similar to
that of the integrable $\G$-modules, where the condition of being
semisimple over the $\mf{sl}(2)$ or $\mf{osp}(1|2)$-copy associated
to each non-isotropic simple root  is supplemented by that of  being
a polynomial representation over $\gl(1|1)$ for each odd isotropic
simple root. It also follows that the irreducible modules in this
semisimple category admit BGG type resolutions in terms of
Verma modules. For the orthosymplectic Lie superalgebras such modules are the so-called oscillator modules. In general, we can define a notion of integrable
$\SG_n$-modules for $n\in\N$, even if $\SG_n$ does not satisfy the condition \eqref{eq:neq0}, and obtain similar results (see Remark \ref{rem:integrable for general SG_n}).  Unlike the
Kac-Moody Lie algebra case, the notion of integrable modules over
Kac-Moody Lie superalgebras with odd isotropic simple roots is
subtle and there is no universal agreement; for some earlier notion
of integrability see \cite{KW2} for affine superalgebras (also see \cite{S}).

The paper is organized as follows. In Section~ \ref{Basics on g(A)},
we recall necessary background on contragredient and Kac-Moody Lie
superalgebras. In Section \ref{Three Lie superalgebras}, we define
finite-rank Lie superalgebras  $\SG_n$, $\G_n$, and their limits
$\SG$, $\G$ of infinite rank, respectively.  In Section
\ref{Irreducible characters}, we specialize $\SG$ and $\G$, and
introduce parabolic BGG categories $\ov{\mc O}$ and ${\mc O}$ of
$\SG$ and $\G$, respectively. It is proved that the characters of  irreducible
highest weight modules in $\ov{\mc O}$ are determined by
those in ${\mc O}$. In Section \ref{Super duality}, we establish an
equivalence of categories between $\ov{\mc O}$ and $\mc O$, which we
refer to as super duality. The irreducible characters in the
category $\ov{\mc O}_n$ of $\SG_n$-modules (which is a finite-rank
version of $\ov{\mc O}$) are given by Kazhdan-Lusztig polynomials
(for a large class of $\SG_n$ and their highest weight modules). Finally,
in Section \ref{sec:integrable}, we construct a full subcategory of
$\ov{\mc O}_n$, which form  a semisimple tensor category.

\vspace{.2cm} \noindent {\bf Acknowledgment.}
The second and third
authors thank the Institute of Mathematics, Academia Sinica, Taipei,
for its hospitality and support.

\section{Kac-Moody Lie superalgebras}\label{Basics on g(A)}

Let $\Z$, $\N$, and $\Z_+$ stand for the sets of all, positive, and
non-negative integers, respectively. All vector spaces, algebras,
etc., are over the complex field $\C$.

Let $I$ be a finite set with a $\Z_2$-grading $I=I_{\ov{0}}\sqcup
I_{\ov{1}}$. Suppose that $|I|=n$ and $A=(a_{ij})_{i,j\in I}$ is a
complex matrix of rank $\ell$. We let $(\h,\Pi,\Pi^\vee)$ be a
minimal realization of $A$ (cf. \cite[\S 1.1]{K}); that is, (1) $\h$
is a vector space of dimension $2n-\ell$, (2)
$\Pi=\{\alpha_i\,|\,i\in I\,\}$ and $\Pi^\vee=\{\alpha^\vee_i\,|\,
i\in I\,\}$ are linearly independent subsets of $\h^*$ and $\h$,
respectively, satisfying
$\langle\alpha_i,\alpha_j^\vee\rangle=a_{ji}$ for $i,j\in I$.

Let $\widehat{\G}(A)$ be the Lie superalgebra generated by $\h$ and
$\{\, e_i,f_i\,|\,i\in I\,\}$ subject to the following relations:
\begin{align}\label{KM:aux:01}
\begin{split}
& [h,h']=0, \quad \ \ \ \ \ \ \ \ \ [e_i,f_j]=\delta_{ij}\alpha_i^\vee,\\
& [h,e_i]=\langle\alpha_i,h\rangle e_i,\quad [h,f_i]=-\langle\alpha_i,h\rangle f_i,\\
%\begin{cases}
%\bar{0}, & \text{if $i\in I_{\ov{0}}$},\\
%\bar{1}, & \text{if $i\in I_{\ov{1}}$},
%\end{cases}\\
\end{split}
\end{align}
for $i,j\in I$ and $h,h'\in\h$.
%The $\Z_2$-gradation on $\DG(A,\tau)$ is determined by $\deg h=\bar{0}$ for $h\in\h$, and
%$\deg e_i=\deg f_i=\bar{0}$ (resp. $\bar{1}$) for $i\in I_{\ov{0}}$  (resp. $i\in \tau$).
Note that $[\alpha^\vee_j,e_i]=a_{ji} e_i$ and
$[\alpha_j^\vee,f_i]=-a_{ji} f_i$. The parity of each generator is
given by $p(h)=\ov{0}$ for $h\in\mf{h}$, and
$p(e_i)=p(f_i)=\varepsilon$, for $i\in I_{\varepsilon}  \
(\varepsilon\in\Z_2)$.

We have a triangular decomposition
$\widehat{\G}(A)=\widehat{\mf{n}}^+\oplus\h\oplus\widehat{\mf{n}}^-$,
where $\widehat{\mf{n}}^+$ and $\widehat{\mf{n}}^-$ are the
subalgebras generated by the $e_i$'s and the $f_i$'s ($i\in I$),
respectively. Let $\mf{r}_1$ and $\mf{r}_2$ be two ideals of
$\widehat{\G}(A)$ such that $\mf{r}_i\cap\h=0$ for $i=1,2$. Then
$\mf{r}_1+\mf{r}_2$ is also an ideal of $\widehat{\G}(A)$
intersecting $\h$ trivially.  Thus, there exists a unique ideal
$\mf{r}$ that is maximal among the ideals intersecting $\h$
trivially.  Following \cite{K1}, we define
$\G(A):=\widehat{\G}(A)/\mf{r}$ and call it the {\em contragredient
Lie superalgebra} associated with $A$.

Let  $\G=\widehat{\G}(A)/\mf{s}$ be the quotient algebra of
$\widehat{\G}(A)$ by an ideal $\mf{s}\subseteq\mf{r}$. Then $\G$ has
the triangular decomposition $\G=\mf{n}^+\oplus\h\oplus\mf{n}^-$
induced from that of $\widehat{\G}(A)$. Consider the following
conditions on $\G$:
\begin{itemize}
\item[(H)] $\G$ has no non-trivial ideal intersecting $\h$ trivially.

\item[(N)]
For every non-zero root vector $x\in\mf{n}^+$, there exists  $f_i$
such that $[x,f_i]\not=0$;
\item[]
for every non-zero root vector $y\in\mf{n}^-$, there exists   $e_i$
such that $[y,e_i]\not=0$.
\end{itemize}

\begin{lem}\label{lem:H=N}
Let $\G$ be as above. The conditions {\rm (H)} and {\rm (N)} are equivalent.
\end{lem}

\begin{proof}
Suppose that $\G$ satisfies (N) and $\mf{s}_0\subseteq \G$ is a
non-zero ideal intersecting $\h$ trivially. Then we can find a
non-zero root vector $x\in \mf{s}_0$ corresponding to a root
$\alpha$, which has minimal height in $\mf{s}_0$. For definiteness,
assume that $\alpha$ is positive or $x\in \mf{s}_0\cap\mf{n}^+$
since the case when $\alpha$ is negative is analogous.  It is clear
that $x\not=e_i$ for $i\in I$, for otherwise $\mf{s}_0\cap\h\not=0$.
By (N), there exists $f_i$ such that $[x,f_i]\not=0$. But $[x,f_i]$
is a positive root vector in $\mf{s}_0$ of height less than that of
$x$. This contradicts the minimality of the height of $x$. So, $\G$
satisfies (H).

Conversely, suppose that $\G$ satisfies (H). If there exists a
positive root vector $x\in\mf{n}^+$ such that $[x,f_i]=0$ for all
$i\in I$, then the ideal generated by $x$ is a non-zero ideal in
$\mf{n}^+$ and hence intersects $\h$ trivially, contradicting (H).
The argument for $y\in\mf{n}^-$ is similar.  Thus, (H) implies (N).
\end{proof}

\begin{lem}  \label{KM:lem:01}
Let $\G$ be a Lie superalgebra generated by an abelian subalgebra
$\h$ of dimension $2n-\ell$ and  $\{\, e_i, f_i\,|\,i\in I\,\}$.
Suppose that there exist linearly independent vectors
$\{\,\alpha^\vee_i\,|\,i\in I\,\}$ in $\h$ and
$\{\,\alpha_i\,|\,i\in I\,\}$ in $\h^*$  such that the relations
\eqref{KM:aux:01} hold in $\G$. If $\G$ satisfies {\rm (H)}, then
$\G\cong\G(A)$.
\end{lem}

\begin{proof}
It follows by the same proof as in the Lie algebra setting; see the
proof of \cite[Proposition 14.15]{Car}.
\end{proof}

The following was stated in \cite[Section 2.5.1]{K1}.

\begin{prop}  \label{prop:KM:embed}
For $J\subseteq I$, let $B=(a_{ij})_{i,j\in J}$ and
$J_{\varepsilon}=I_{\varepsilon}\cap J$ for $\varepsilon\in \Z_2$.
Then we have a natural inclusion of Lie superalgebras:
$\G(B)\hookrightarrow\G(A)$.
\end{prop}

\begin{proof}
Suppose that $B$ has rank $r$.  By \cite[Exercise 1.2]{K} we can find $\h'\subset \h$ with $\dim\h'=2|J|-r$ such that $\h'$ contains $\{\,\alpha^\vee_j\,|\,j\in J\,\}$ and $\{\,\alpha_j|_{\h'}\,|\,j\in J\,\}$ is linearly independent, so that we have a minimal realization of $B$. Let $\G$ be the subalgebra of $\G(A)$ generated by $\h'$ and $\{\,e_i , f_i\,|\, i\in J\,\}$. It suffices to prove that $\G$ satisfies the condition
(N).  First suppose that $x$ is a positive root vector in $\G$ such
that $[x,f_j]=0$  for all $j\in J$.  Since  $x$ is generated by the
$e_j$'s with $j\in J$, we must have $[x,f_i]=0$ for $i\not\in J$,
which implies that $x=0$. The case when $x$ is a negative root
vector is analogous. Now we apply Lemmas \ref{lem:H=N} and
\ref{KM:lem:01} and conclude that $\G\cong\G(B)$.
\end{proof}

Recall the following, which describes the effect of an odd
reflection on a fundamental system of a contragredient Lie
superalgebra (cf.~\cite{PS}).

\begin{prop}\cite[Lemma 1.2]{KW2}\label{prop:odd:reflec}
Suppose that $A=(a_{ij})$ satisfies $a_{ij}=0$ if and only if
$a_{ji}=0$. Let $\alpha_s$ be an odd isotropic simple root of the
fundamental system $\Pi$ associated with $A$ of a contragredient Lie
superalgebra $\G(A)$. The simple roots and coroots of the
fundamental system $\Pi'$, obtained by applying the odd reflection
associated with $\alpha_s$ to $\Pi$, are precisely as follows:
\begin{itemize}
\item[(1)] $-\alpha_s\in\Pi'$ with corresponding coroot $-\alpha^\vee_s$,

\item[(2)]
$\alpha_i\in\Pi'$ with corresponding coroot $\alpha^\vee_i$ for
$i\not=s$ with $a_{is}=0$,

\item[(3)] $\alpha_i+\alpha_s\in\Pi'$ with corresponding coroot
$\alpha^\vee_i+\frac{a_{is}}{a_{si}}\alpha^\vee_s$  for $i\not=s$ with $a_{is}\not=0$.
\end{itemize}
\end{prop}

An integer matrix $A=(a_{ij})_{i,j\in I}$ is called a {\em super generalized Cartan
matrix (abbreviated} SGCM) if it satisfies the following conditions:
\begin{itemize}
\item[(C0)] $a_{ii}=2$  for all $i\in I_{\ov{0}}$;

\item[(C1)] $a_{ii}=2$ or $0$ for all  $i\in I_{\ov{1}}$;

\item[(C2)]
If $a_{ii}=2$ then, for all $j\neq i$, $a_{ij} \in -\Z_+$
for $i\in I_{\ov{0}}$ and $a_{ij}\in -2\Z_+$ for $i\in I_{\ov{1}}$;

\item[(C3)] $a_{ij}=0$ if and only if $a_{ji}=0$.
\end{itemize}

Following \cite{vdL} we call the corresponding contragredient Lie
superalgebra $\G(A)$ the {\em Kac-Moody Lie superalgebra} associated
with a SGCM $A$.
A SGCM $A=(a_{ij})_{i,j\in I}$
and the corresponding Kac-Moody Lie superalgebra $\G(A)$ are called {\em anisotropic}
if $A$ satisfies
the following condition (stronger than (C0) and (C1) above):
\vspace{.1cm}

$\text{(C$1'$)} \;\; a_{ii}=2 \text{ for all  } i\in I.$

\noindent Note that it is possible that $a_{ij}>0$ for $i\neq j$ for
a SGCM, which occurs only when
$a_{ii}=0$. In particular, an anisotropic  Kac-Moody Lie
superalgebra does not have a positive off-diagonal entry in its SGCM.  Also, note that if
$I_{\ov{1}}=\emptyset$, then $A$ is a generalized Cartan matrix (simply GCM)
as defined in \cite[\S 1.1]{K}, and $\G(A)$ is the {\em Kac-Moody
Lie algebra} associated with $A$.

\begin{rem}
We would like to point out that the definition of a Kac-Moody Lie superalgebra in \cite{vdL}, and which is the one used in this paper, differs from the one in \cite{S}  (see also \cite{Hoyt,HS}). Indeed, the Kac-Moody Lie superalgebra in this paper in the terminology of \cite{S} is referred to as a regular admissible contragredient Lie superalgebra.

\end{rem}

\begin{example}\label{ex:except:simples}
Consider the three simple finite-dimensional exceptional Lie
superalgebras $G(3)$, $F(3|1)$ (sometimes denoted by $F(4)$), and
$D(2|1,\alpha)$ for $\alpha\in\C\setminus\{0,-1\}$. According to
\cite[Proposition 2.5.4]{K1}, we can choose the following
SGCM's for their distinguished
fundamental systems (in contrast to the matrices in loc.~cit.~we
have rescaled the odd isotropic simple root by $-1$):
\begin{align*}
\begin{pmatrix}
\ \ 0&-1&\ \ 0\ \\
-1&\ \ 2&\!\!-3\\
\ \  0&-1&\ \ 2\
\end{pmatrix},
\begin{pmatrix}
\ \ 0&-1&\ \ 0&\ \ 0\ \\
-1&\ \ 2&-2&\ \ 0\ \\
\ \ 0&-1&\ \ 2&\!\!-1 \\
\ \ 0&\ \ 0&-1&\ \ 2\
\end{pmatrix},
\begin{pmatrix}
\ \ 0&-1&\!\!-\alpha\\
-1&\ \ 2&\ \ 0\ \\
-1&\ \ 0&\ \ 2\
\end{pmatrix},
\end{align*}
where only the first row index has parity $\ov{1}$. Note that if we
replace the zero diagonal entries in the above matrices by $2$ and
change the parity of the corresponding row index, then the above
matrices become symmetrizable generalized Cartan matrices so that
the resulting Lie algebras are symmetrizable Kac-Moody Lie algebras.
(In the case of $D(2|1,\alpha)$, we need to assume that
$\alpha\in\N$.)
\end{example}

\begin{rem}\label{rem:Serre relations}
Suppose that $A=(a_{ij})_{i,j\in I}$ is a SGCM.  Then in $\G(A)$ we have the following relations for $i,j\in
I$ with $i\neq j$ (\cite[Proposition 3.3]{vdL}):
\begin{equation*}
\begin{split}
{\rm ad}(e_i)^{1-a_{ij}}(e_j) &=0, \ \ \ \text{if $a_{ii}=2$},\\
{\rm ad}(e_i)^{m(j)}(e_j)&=0, \ \ \ \text{if $a_{ii}=0$},\\
\end{split}
\end{equation*}
where $m(j)=1$ if $a_{ij}=0$ and $m(j)=2$ if $a_{ij}\neq 0$.

\end{rem}

\section{The Lie superalgebras $\tilde{\G}$, $\SG$, and $\G$}
 \label{Three Lie superalgebras}

Let $I$ be a $\Z_2$-graded set and let $A=(a_{ij})_{i,j\in I}$ be a
matrix with $a_{ij}\in \Z$ and $a_{ii}=2$ or $0$ for $i,j\in I$.  We
draw the corresponding Dynkin diagram as follows: The vertices of
the Dynkin diagram are parametrized by $I$, and denoted by
$\bigcirc$ (resp. $
\begin{picture}(2,2)\setlength{\unitlength}{0.14in}
\put(0.3,0.3){\circle*{1}}
\end{picture}
$\ \ )
when $a_{ii}=2$ with $i  \in I_{\ov{0}}$ (resp. $i\in
I_{\ov{1}}$), and by $\bigotimes$ when $a_{ii}=0$, for $i\in I_{\ov{1}}$.
Furthermore, for two vertices $i,j$ with $a_{ij}\not=0$ or $a_{ji}\neq 0$, we draw an
edge connecting them and label it with the pair of integers
$(a_{ij},a_{ji})$ as follows:
\begin{center}
\setlength{\unitlength}{0.25in}
\begin{picture}(4,2)
\put(0.45,0.85){$\bullet$}
\put(3.2,0.85){$\bullet$}
\put(0.45,0.25){$i$}
\put(3.2,0.25){$j$}
\put(1,1){\line(1,0){2}}
\put(1.1,1.2){$^{(a_{ij}\,,\,a_{ji})}$}
\end{picture}
\end{center}
where $\bullet$ denotes one of $\bigcirc$, $\ \begin{picture}(2,2)\setlength{\unitlength}{0.14in} \put(0.3,0.3){\circle*{1}}\end{picture}\ \ $, or $\bigotimes$.

For $r\in T_\infty:=\{-1\}\cup \frac{1}{2}\mathbb{N}$, we consider a
SGCM $(a_{st})$ labelled by the (tail)
set $T_r:=\{-1,\frac{1}{2},1,\frac{3}{2},\ldots,r\}$
%\begin{equation*}
%A^{\texttt t}_r:=\begin{pmatrix}
%\ 0&\ \ 1&\ \ 0&\ \ 0&\cdots \ \\
%\ 1&\ \ 0&-1&\ \ 0&\cdots \ \\
%\ 0&-1&\ \ 0& \ \ 1&\cdots \ \\
%\ 0& \ \ 0&\ \ 1&\ \ 0& \cdots\  \\
%\ \vdots&\ \ \vdots&\ \ \vdots&\vdots & \ddots\ \\
%\end{pmatrix}.
%\end{equation*}
with $a_{-1\,\frac{1}{2}}=a_{\frac{1}{2}\, -1}=1$, $a_{s\,
s+\hf}=a_{s+\hf\,s}=(-1)^{2s}$ for $s>0$, and $a_{st}=0$ elsewhere.
Also, by definition the parity of every $s\in T_r$ is $\ov{1}$. Let
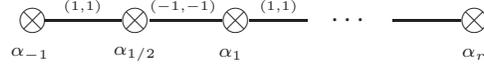
\begin{center}
\hskip -4cm \setlength{\unitlength}{0.25in}
\begin{picture}(24,2)
\put(10.25,1){\makebox(0,0)[c]{\Large $\otimes$}}
\put(12.4,1){\makebox(0,0)[c]{\Large $\otimes$}}
\put(16.9,1){\makebox(0,0)[c]{$\cdots$}}
\put(10.5,1){\line(1,0){1.6}}
\put(10.9,1.3){\tiny{$_{(1,1)}$}}
\put(12.7,1){\line(1,0){1.5}}
\put(12.7,1.3){\tiny{$_{(-1,-1)}$}}
\put(14.8,1){\line(1,0){1.2}}
\put(15,1.3){\tiny{$_{(1,1)}$}}
\put(17.8,1){\line(1,0){1.4}}
\put(14.5,1){\makebox(0,0)[c]{\Large $\otimes$}}
\put(19.5,1){\makebox(0,0)[c]{\Large $\otimes$}}
\put(10.2,0.3){\makebox(0,0)[c]{\tiny $\alpha_{-1}$}}
\put(12.4,0.3){\makebox(0,0)[c]{\tiny $\alpha_{1/2}$}}
\put(14.4,0.3){\makebox(0,0)[c]{\tiny $\alpha_{1}$}}
\put(19.5,0.3){\makebox(0,0)[c]{\tiny $\alpha_{r}$}}
\end{picture}
\end{center}
be the (tail) Dynkin diagram associated to this SGCM $(a_{st})$ with a fundamental system
$\{\,\alpha_{-1},\alpha_{\hf},\alpha_1,\ldots,\alpha_r\}$
parametrized by $T_r$.

Suppose that a SGCM $B$ is given. We
denote its Dynkin diagram by \framebox{$\mf{B}$} and its fundamental
system by $\Pi_B$. We consider the following Dynkin diagram, which
we call the {\it head diagram}, by connecting a vertex $\bigotimes$
also indexed by $\alpha_{-1}$ to some vertices $\gamma_j$ in
\framebox{$\mf{B}$} with labels $(b_j, c_j)$ $(1\le j\le p)$:
 \vskip -2cm
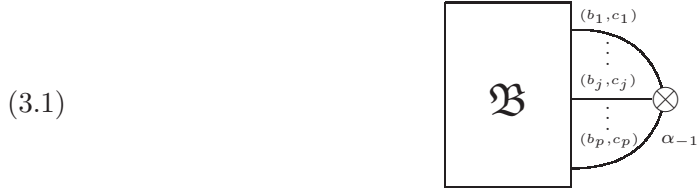
\begin{equation}\label{Head diagram}
\text{ \hskip -2cm \setlength{\unitlength}{0.22in}
\begin{picture}(12,6)
\put(10.25,0.3){\makebox(0,0)[c]{\Large $\otimes$}}
\put(8,.3){\line(1,0){1.9}}
\qbezier(10.15,0.6)(9.7,2)(8,1.95)
\qbezier(10.15,0)(9.7,-1.35)(8,-1.35)
\put(8.2,2.3){\tiny{$_{(b_1,c_1)}$}}
\put(8.8,1.9){\tiny{$_{\vdots}$}}
\put(8.2,0.75){\tiny{$_{(b_j,c_j)}$}}
\put(8.8,0.35){\tiny{$_{\vdots}$}}
\put(8.2,-0.65){\tiny{$_{(b_p,c_p)}$}}
\put(8,-1.8){\line(0,1){4.4}}
\put(5,-1.8){\line(0,1){4.4}}
\put(8,2.6){\line(-1,0){3}}
\put(8,-1.8){\line(-1,0){3}}
\put(6.5,.3){\makebox(0,0)[c]{{\LARGE$\mf{B}$}}}
\put(10.6,-0.7){\makebox(0,0)[c]{\tiny $\alpha_{-1}$}}
\end{picture}}
\end{equation}
\vskip 1.2cm \noindent
Here,
$b_j=\langle\alpha_{-1},\gamma^\vee_j\rangle$, and
$c_j=\langle\gamma_j,\alpha^\vee_{-1}\rangle$ for $1\leq j\leq p$,
where it is understood that $\alpha_{-1}^{\vee}$ and $\gamma_j^\vee$
(respectively $\alpha_{-1}$ and $\gamma_j$) are the associated simple
coroots (respectively simple roots) in a minimal realization of \eqref{Head diagram}.
We denote by $A^{\texttt{hd}}$ the matrix
corresponding to the head diagram \eqref{Head diagram}.
Throughout the paper, we shall always
assume that $A^{\texttt{hd}}$ is a SGCM, or equivalently, for $1\leq j\leq p$,
\begin{align}\label{eq:pseudoCartan for hd}\tag{{\bf A}}
\begin{split}
&b_j,c_j\in\Z\setminus\{0\} ,\\
&\text{$b_j \in -\Z_+$,\ \
if  $\gamma_j
=\bigcirc$,}\\
&\text{$b_j \in -2\Z_+$,\ \
if  $\gamma_j
=\begin{picture}(2,2)\setlength{\unitlength}{0.14in}
\put(0.3,0.3){\circle*{1}}
\end{picture}$\ \ .}
\end{split}
\end{align}

\begin{example}\label{example:affineosp}
Suppose that \framebox{$\mf{B}$} is a  disjoint union of two
classical Dynkin diagrams of  type $C$ and type $B$:
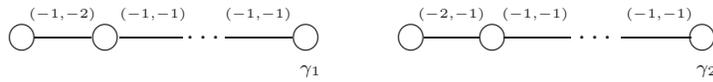
\begin{center}
\hskip -4cm \setlength{\unitlength}{0.25in}
\begin{picture}(22,2)
\put(6.45,1){\makebox(0,0)[c]{$\bigcirc$}}
\put(8.2,1){\makebox(0,0)[c]{$\bigcirc$}}
\put(10.3,1){\makebox(0,0)[c]{$\cdots$}}
\put(12.4,1){\makebox(0,0)[c]{$\bigcirc$}}
\put(18.5,1){\makebox(0,0)[c]{$\cdots$}}
\put(16.3,1){\makebox(0,0)[c]{$\bigcirc$}}
\put(14.6,1){\makebox(0,0)[c]{$\bigcirc$}}
\put(20.7,1){\makebox(0,0)[c]{$\bigcirc$}}
\put(6.7,1){\line(1,0){1.2}}%{\Large $\Longrightarrow$}
\put(8.5,1){\line(1,0){1.3}}
\put(10.7,1){\line(1,0){1.4}}
\put(6.6,1.5){\tiny{$_{(-1,-2)}$}}
\put(8.5,1.5){\tiny{$_{(-1,-1)}$}}
\put(10.7,1.5){\tiny{$_{(-1,-1)}$}}
\put(14.75,1.5){\tiny{$_{(-2,-1)}$}}
\put(16.5,1.5){\tiny{$_{(-1,-1)}$}}
\put(19.1,1.5){\tiny{$_{(-1,-1)}$}}
\put(14.87,1){\line(1,0){1.15}}%{\Large $\Longleftarrow$}
\put(16.55,1){\line(1,0){1.4}}
\put(19,1){\line(1,0){1.45}}
%\put(19.1,0.8){\Large $\Longrightarrow$}
%

%
\put(12.5,0.3){\makebox(0,0)[c]{\tiny $\gamma_{1}$}}
\put(20.8,0.3){\makebox(0,0)[c]{\tiny $\gamma_2$}}
\end{picture}
\end{center}
\noindent Incorporating the lengths of the roots into the diagram and removing the labels we get the more familiar forms:
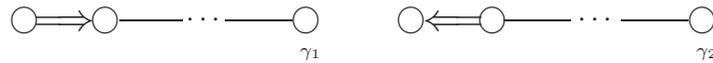
\begin{center}
\hskip -4cm \setlength{\unitlength}{0.25in}
\begin{picture}(22,2)
\put(6.5,1){\makebox(0,0)[c]{$\bigcirc$}}
\put(8.2,1){\makebox(0,0)[c]{$\bigcirc$}}
\put(10.3,1){\makebox(0,0)[c]{$\cdots$}}
\put(12.4,1){\makebox(0,0)[c]{$\bigcirc$}}
\put(18.5,1){\makebox(0,0)[c]{$\cdots$}}
\put(16.3,1){\makebox(0,0)[c]{$\bigcirc$}}
\put(14.6,1){\makebox(0,0)[c]{$\bigcirc$}}
\put(20.7,1){\makebox(0,0)[c]{$\bigcirc$}}
\put(6.7,0.8){\Large $\Longrightarrow$}
\put(8.5,1){\line(1,0){1.3}}
\put(10.7,1){\line(1,0){1.4}}
\put(14.85,0.8){\Large $\Longleftarrow$}
\put(16.55,1){\line(1,0){1.4}}
\put(19,1){\line(1,0){1.4}}
%\put(19.1,0.8){\Large $\Longrightarrow$}
%
%
\put(12.5,0.3){\makebox(0,0)[c]{\tiny $\gamma_{1}$}}
\put(20.8,0.3){\makebox(0,0)[c]{\tiny $\gamma_2$}}
\end{picture}
\end{center}
\noindent Here we regard $\gamma_1$, and $\gamma_2$  as the two end vertices
of \framebox{$\mf{B}$}. We connect \framebox{$\mf{B}$} to a vertex
$\alpha_{-1}$ at $\gamma_1$ and $\gamma_2$ with labels
$(b_1,c_1)=(-1,-1)$ and $(b_2,c_2)=(-1,1)$, respectively, so that the resulting
head diagram is of the form:
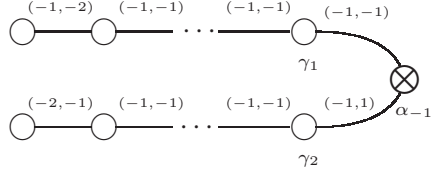
\begin{center}
\hskip 0cm \setlength{\unitlength}{0.25in}
\begin{picture}(22,4)
\put(6.5,1){\makebox(0,0)[c]{$\bigcirc$}}
\put(8.2,1){\makebox(0,0)[c]{$\bigcirc$}}
\put(10.2,1){\makebox(0,0)[c]{$\cdots$}}
\put(12.4,1){\makebox(0,0)[c]{$\bigcirc$}}
\put(10.25,3){\makebox(0,0)[c]{$\cdots$}}
\put(8.2,3){\makebox(0,0)[c]{$\bigcirc$}}
\put(6.5,3){\makebox(0,0)[c]{$\bigcirc$}}
\put(12.4,3){\makebox(0,0)[c]{$\bigcirc$}}
\qbezier(12.65,3)(14,3)(14.4,2.3)
\qbezier(12.65,1)(14,1)(14.4,1.7)
\put(14.5,2){\makebox(0,0)[c]{$\bigotimes$}}
\put(6.75,1){\line(1,0){1.2}}%{\Large $\Longleftarrow$}
\put(8.45,1){\line(1,0){1.2}}
\put(10.6,1){\line(1,0){1.5}}
\put(6.6,1.5){\tiny{$_{(-2,-1)}$}}
\put(8.5,1.5){\tiny{$_{(-1,-1)}$}}
\put(10.75,1.5){\tiny{$_{(-1,-1)}$}}
\put(6.6,3.5){\tiny{$_{(-1,-2)}$}}
\put(8.5,3.5){\tiny{$_{(-1,-1)}$}}
\put(10.75,3.5){\tiny{$_{(-1,-1)}$}}
\put(6.75,3){\line(1,0){1.2}}%{\Large $\Longrightarrow$}
\put(8.45,3){\line(1,0){1.2}}
\put(10.6,3){\line(1,0){1.5}}
%\put(19.1,0.8){\Large $\Longrightarrow$}
%
\put(12.8,3.4){\tiny{$_{(-1,-1)}$}}
\put(12.8,1.5){\tiny{$_{(-1,1)}$}}
\put(12.5,0.3){\makebox(0,0)[c]{\tiny $\gamma_{2}$}}
\put(12.5,2.3){\makebox(0,0)[c]{\tiny $\gamma_1$}}
\put(14.7,1.3){\makebox(0,0)[c]{\tiny $\alpha_{-1}$}}
\end{picture}
\end{center}
\noindent Incorporating the lengths of the roots (and removing the labels) we can stretch it into its more familiar form, which is the Dynkin diagram of the affine Lie superalgebra of type $B$:
{\begin{center}
\hskip -4cm \setlength{\unitlength}{0.25in}
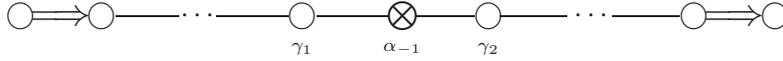
\begin{picture}(24,2)
\put(6.5,1){\makebox(0,0)[c]{$\bigcirc$}}
\put(8.2,1){\makebox(0,0)[c]{$\bigcirc$}}
\put(10.26,1){\makebox(0,0)[c]{$\cdots$}}
\put(12.4,1){\makebox(0,0)[c]{$\bigcirc$}}
\put(16.3,1){\makebox(0,0)[c]{$\bigcirc$}}
\put(18.5,1){\makebox(0,0)[c]{$\cdots$}}
\put(6.7,0.8){\Large $\Longrightarrow$}
\put(20.8,0.8){\Large $\Longrightarrow$}
\put(22.3,1){\makebox(0,0)[c]{$\bigcirc$}}
%
%
%\put(6.6,1.5){\tiny{$_{(-1,-2)}$}}
\put(8.5,1){\line(1,0){1.3}}
\put(10.6,1){\line(1,0){1.5}}
%\put(8.5,1.5){\tiny{$_{(-1,-1)}$}}
%\put(10.7,1.5){\tiny{$_{(-1,-1)}$}}
\put(12.7,1){\line(1,0){1.5}}
%\put(12.7,1.5){\tiny{$_{(-1,-1)}$}}
\put(14.8,1){\line(1,0){1.2}}
%\put(14.9,1.5){\tiny{$_{(1,-1)}$}}
%\put(16.6,1.5){\tiny{$_{(-1,-1)}$}}
%\put(18.8,1.5){\tiny{$_{(-1,-1)}$}}
\put(16.55,1){\line(1,0){1.4}}
\put(18.9,1){\line(1,0){1.4}}
%\put(20.7,1.5){\tiny{{$_{(-1,-2)}$}}}
%
\put(14.5,1){\makebox(0,0)[c]{$\bigotimes$}}
\put(20.6,1){\makebox(0,0)[c]{$\bigcirc$}}
\put(14.5,0.3){\makebox(0,0)[c]{\tiny $\alpha_{-1}$}}
\put(12.4,0.3){\makebox(0,0)[c]{\tiny $\gamma_{1}$}}
\put(16.3,0.3){\makebox(0,0)[c]{\tiny $\gamma_2$}}
\end{picture}
\end{center}}

\noindent

Similarly, one constructs the Dynkin diagram of the affine Lie
superalgebra of type $D$ as a head diagram \eqref{Head diagram} with
\framebox{$\mf{B}$} being a disjoint union of two classical Dynkin
diagrams of type $D$ and type $C$.
\end{example}

For $r\in T_\infty$, we consider a new Dynkin diagram by merging the
head diagram with the tail diagram parametrized by $T_r$ as follows
(or by identifying the $\alpha_{-1}$  in the head and tail
diagrams):\vskip -2.1cm
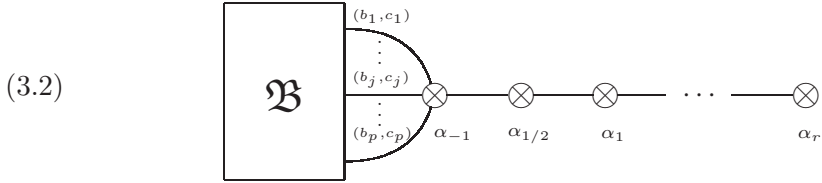
\begin{equation}\label{DG diagram of finite rank}
\text{
\hskip -2cm \setlength{\unitlength}{0.21in}
\begin{picture}(23,6)
\put(10.25,0){\makebox(0,0)[c]{\Large $\otimes$}}
\put(12.4,0){\makebox(0,0)[c]{\Large $\otimes$}}
\put(16.9,0){\makebox(0,0)[c]{$\cdots$}}
\put(8,0){\line(1,0){1.9}}
\qbezier(10.15,0.3)(9.7,1.7)(8,1.65)
\qbezier(10.15,-0.3)(9.7,-1.65)(8,-1.65)
\put(10.55,0){\line(1,0){1.5}}
\put(12.7,0){\line(1,0){1.5}}
\put(14.8,0){\line(1,0){1.2}}
\put(17.65,0){\line(1,0){1.5}}
\put(8.2,2){\tiny{$_{(b_1,c_1)}$}}
\put(8.8,1.6){\tiny{$_{\vdots}$}}
\put(8.2,0.45){\tiny{$_{(b_j,c_j)}$}}
\put(8.8,0.05){\tiny{$_{\vdots}$}}
\put(8.2,-0.95){\tiny{$_{(b_p,c_p)}$}}
\put(14.5,0){\makebox(0,0)[c]{\Large $\otimes$}}
\put(19.5,0){\makebox(0,0)[c]{\Large $\otimes$}}
\put(8,-2.1){\line(0,1){4.4}}
\put(5,-2.1){\line(0,1){4.4}}
\put(8,2.3){\line(-1,0){3}}
\put(8,-2.1){\line(-1,0){3}}
\put(6.5,0){\makebox(0,0)[c]{{\LARGE$\mf{B}$}}}
\put(10.7,-1){\makebox(0,0)[c]{\tiny $\alpha_{-1}$}}
\put(12.6,-1){\makebox(0,0)[c]{\tiny $\alpha_{1/2}$}}
\put(14.7,-1){\makebox(0,0)[c]{\tiny $\alpha_{1}$}}
\put(19.6,-1){\makebox(0,0)[c]{\tiny $\alpha_{r}$}}
\end{picture}}
\end{equation}
\vskip 1.3cm
Let $\DG_r^\circ$ be the Kac-Moody Lie superalgebra associated with
the above Dynkin diagram, whose fundamental system is $\Pi_B \cup
\{\,\alpha_i \mid i\in T_r\,\}$.
Since $\DG_r^\circ\subseteq\DG_{r+\hf}^\circ$ by Proposition
\ref{prop:KM:embed}, we have a well-defined Lie superalgebra
$\DG^\circ=\bigcup_{r}\DG_r^\circ$. It is easy to see that
$\DG^\circ$ satisfies (N) and hence (H) by Lemma \ref{lem:H=N}.
Thus, by abuse of terminology one may regard $\DG^\circ$ as the
Kac-Moody Lie superalgebra (of infinite rank) associated with the
following Dynkin diagram:\vskip -1.8cm

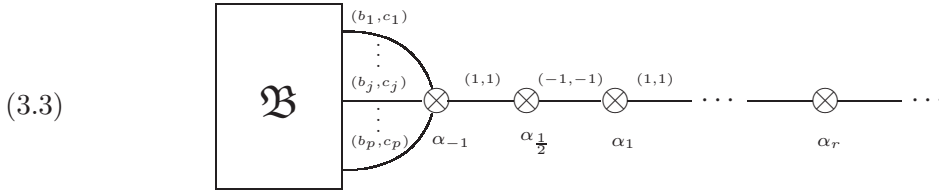
\begin{equation}\label{DG diagram}
\text{\hskip -2cm \setlength{\unitlength}{0.22in}
\begin{picture}(24,5.5)
\put(10.25,0.3){\makebox(0,0)[c]{\Large $\otimes$}}
\put(12.4,0.3){\makebox(0,0)[c]{\Large $\otimes$}}
\put(17,.3){\makebox(0,0)[c]{$\cdots$}}
\put(22,.3){\makebox(0,0)[c]{$\cdots$}}
\put(8,.3){\line(1,0){1.9}}
\qbezier(10.15,0.6)(9.7,2)(8,1.95)
\qbezier(10.15,0)(9.7,-1.35)(8,-1.35)
\put(10.52,.3){\line(1,0){1.55}}
\put(12.7,.3){\line(1,0){1.45}}
\put(14.8,.3){\line(1,0){1.5}}
\put(17.65,.3){\line(1,0){1.55}}
\put(19.8,.3){\line(1,0){1.5}}
\put(8.2,2.3){\tiny{$_{(b_1,c_1)}$}}
\put(8.8,1.9){\tiny{$_{\vdots}$}}
\put(8.2,0.75){\tiny{$_{(b_j,c_j)}$}}
\put(8.8,0.35){\tiny{$_{\vdots}$}}
\put(8.2,-0.65){\tiny{$_{(b_p,c_p)}$}}
\put(14.5,.3){\makebox(0,0)[c]{\Large $\otimes$}}
\put(19.5,.3){\makebox(0,0)[c]{\Large $\otimes$}}
\put(8,-1.8){\line(0,1){4.4}}
\put(5,-1.8){\line(0,1){4.4}}
\put(8,2.6){\line(-1,0){3}}
\put(8,-1.8){\line(-1,0){3}}
\put(6.5,.3){\makebox(0,0)[c]{{\LARGE$\mf{B}$}}}
\put(10.6,-0.7){\makebox(0,0)[c]{\tiny $\alpha_{-1}$}}
\put(12.6,-0.7){\makebox(0,0)[c]{\tiny $\alpha_{\hf}$}}
\put(14.7,-0.7){\makebox(0,0)[c]{\tiny $\alpha_{1}$}}
\put(19.6,-0.7){\makebox(0,0)[c]{\tiny $\alpha_{r}$}}
\put(10.9,0.8){\tiny{$_{(1,1)}$}}
\put(12.65,0.8){\tiny{$_{(-1,-1)}$}}
\put(15,0.8){\tiny{$_{(1,1)}$}}
\end{picture}}
\end{equation}\vskip 1.3cm
\noindent We define $\DG:=\DG^\circ \oplus \C d$, which is an extension of $\DG^\circ$ by an outer derivation of $\DG^\circ$ with $[d,e_{\alpha_{-1}}]=-e_{\alpha_{-1}}$ and $[d,f_{\alpha_{-1}}]=f_{\alpha_{-1}}$ but $[d,x]=0$ for the other generators $x$ of $\DG^\circ$. Here $e_{\alpha_{-1}}$ and $f_{\alpha_{-1}}$ denote the positive and negative generators in $\DG^\circ$ corresponding to $\alpha_{-1}$, respectively. Note that ${\rm ad}(d)$ is well-defined since the ideals which define $\DG^\circ_r$ are graded.
Also set $\DG_r:=\DG_r^\circ \oplus \C d$. Let $\wt{\Pi} :=\Pi_B
\cup \{\,\alpha_i\mid i\in T_\infty\,\}$ denote  this fundamental
system of $\DG$, and let $\wt{\h}$ be the Cartan subalgebra of
$\DG$, which is a direct sum of the Cartan subalgebra of $\DG^\circ$ and $\mathbb{C}d$. \vskip 2mm

Set
$$
\beta_{-1}=\alpha_{-1}+\alpha_{\hf}, \ \
\beta_{r}=\alpha_r+\alpha_{r+\hf}\ \ \ \ (r\in\tfrac{1}{2}\N).$$

For $n\in \N$, consider the following Dynkin diagram:
\vskip -1.7cm
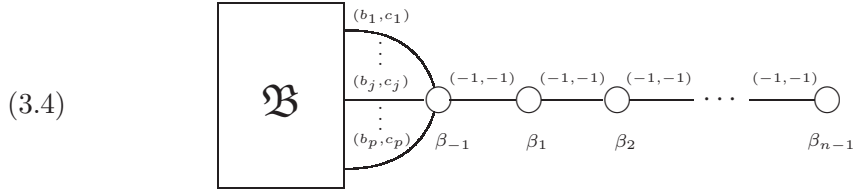
\begin{equation}\label{G diagram}
\text{\hskip -2cm \setlength{\unitlength}{0.22in}
\begin{picture}(24,5.5)
\put(10.25,0.3){\makebox(0,0)[c]{$\bigcirc$}}
\put(12.4,0.3){\makebox(0,0)[c]{$\bigcirc$}}
\put(17,.3){\makebox(0,0)[c]{$\cdots$}}
\put(8,.3){\line(1,0){1.9}}
\qbezier(10.15,0.6)(9.7,2)(8,1.95)
\qbezier(10.15,0)(9.7,-1.35)(8,-1.35)
\put(10.52,.3){\line(1,0){1.55}}
\put(12.7,.3){\line(1,0){1.45}}
\put(14.8,.3){\line(1,0){1.5}}
\put(17.65,.3){\line(1,0){1.55}}
\put(8.2,2.3){\tiny{$_{(b_1,c_1)}$}}
\put(8.8,1.9){\tiny{$_{\vdots}$}}
\put(8.2,0.75){\tiny{$_{(b_j,c_j)}$}}
\put(8.8,0.35){\tiny{$_{\vdots}$}}
\put(8.2,-0.65){\tiny{$_{(b_p,c_p)}$}}
\put(14.5,.3){\makebox(0,0)[c]{$\bigcirc$}}
\put(19.5,.3){\makebox(0,0)[c]{$\bigcirc$}}
\put(8,-1.8){\line(0,1){4.4}}
\put(5,-1.8){\line(0,1){4.4}}
\put(8,2.6){\line(-1,0){3}}
\put(8,-1.8){\line(-1,0){3}}
\put(6.5,.3){\makebox(0,0)[c]{{\LARGE$\mf{B}$}}}
\put(10.6,-0.7){\makebox(0,0)[c]{\tiny $\beta_{-1}$}}
\put(12.6,-0.7){\makebox(0,0)[c]{\tiny $\beta_{1}$}}
\put(14.7,-0.7){\makebox(0,0)[c]{\tiny $\beta_{2}$}}
\put(19.6,-0.7){\makebox(0,0)[c]{\tiny $\beta_{n-1}$}}
\put(10.5,0.8){\tiny{$_{(-1,-1)}$}}
\put(12.65,0.8){\tiny{$_{(-1,-1)}$}}
\put(14.7,0.8){\tiny{$_{(-1,-1)}$}}
\put(17.7,0.8){\tiny{$_{(-1,-1)}$}}
%\put(19.7,0.7){\tiny{$_{(-1,-1)}$}}
\end{picture}}
\end{equation}\vskip 1.3cm
\noindent
When $n=1$ we mean the diagram \eqref{G diagram} with no
$\beta_i$'s for $i\ge 1$. We note that the corresponding
contragredient Lie superalgebra, denoted by $\G^\circ_n$, is a
Kac-Moody Lie superalgebra if and only if $c_j\in-\Z$ for
$1\leq j\leq p$. Similarly, let $\G^\circ =\bigcup_n \G^\circ_n$ and
let $\G:=\G^\circ\oplus\C d$ be its extension by an outer derivation $d$ with $[d,e_{\beta_{-1}}]=-e_{\beta_{-1}}$, $[d,f_{\beta_{-1}}]=f_{\beta_{-1}}$ and $[d,x]=0$ for the other generators. We note that $d$ is possibly an inner derivation plus a central element. Set
$\G_n:=\G^\circ_n \oplus\C d$. Let $\Pi :=\Pi_B \cup \{\beta_{-1}\}
\cup \{\,\beta_j \mid j \in  \N \,\}$ denote this fundamental system
of $\G$, and let $\h$ be the Cartan subalgebra of $\G$.
%Letting ${\G}^{\texttt t}$ be the infinite-rank contragredient
%Lie superalgebra corresponding to the
%subdiagram consisting the vertices $\{\,\beta_{m}\,|\,m\in\mathbb{N}\,\}$,
%we let ${\mf l}$ be the Levi subalgebra ${\h}\oplus{\G}^{\texttt t}$.

On the other hand, for $n\in \N\cup\{-1\}$, let $\SG^\circ_n$ be the
Kac-Moody Lie superalgebra associated with the Dynkin diagram below:
\vskip -1.7cm
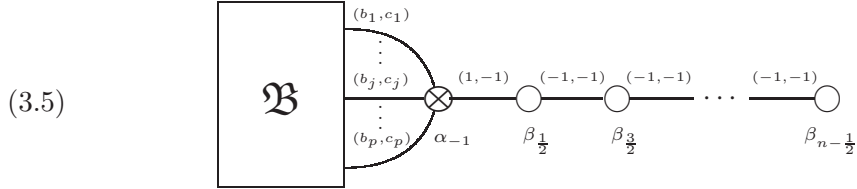
\begin{equation}\label{SG diagram}
\text{\hskip -2cm \setlength{\unitlength}{0.22in}
\begin{picture}(24,5.5)
\put(10.25,0.3){\tiny\makebox(0,0)[c]{$\bigotimes$}}
\put(12.4,0.3){\makebox(0,0)[c]{$\bigcirc$}}
\put(17,.3){\makebox(0,0)[c]{$\cdots$}}
\put(8,.3){\line(1,0){1.9}}
\qbezier(10.15,0.6)(9.7,2)(8,1.95)
\qbezier(10.15,0)(9.7,-1.35)(8,-1.35)
\put(10.52,.3){\line(1,0){1.55}}
\put(12.7,.3){\line(1,0){1.45}}
\put(14.8,.3){\line(1,0){1.5}}
\put(17.65,.3){\line(1,0){1.55}}
\put(8.2,2.3){\tiny{$_{(b_1,c_1)}$}}
\put(8.8,1.9){\tiny{$_{\vdots}$}}
\put(8.2,0.75){\tiny{$_{(b_j,c_j)}$}}
\put(8.8,0.35){\tiny{$_{\vdots}$}}
\put(8.2,-0.65){\tiny{$_{(b_p,c_p)}$}}
\put(14.5,.3){\makebox(0,0)[c]{$\bigcirc$}}
\put(19.5,.3){\makebox(0,0)[c]{$\bigcirc$}}
\put(8,-1.8){\line(0,1){4.4}}
\put(5,-1.8){\line(0,1){4.4}}
\put(8,2.6){\line(-1,0){3}}
\put(8,-1.8){\line(-1,0){3}}
\put(6.5,.3){\makebox(0,0)[c]{{\LARGE$\mf{B}$}}}
\put(10.6,-0.7){\makebox(0,0)[c]{\tiny $\alpha_{-1}$}}
\put(12.6,-0.7){\makebox(0,0)[c]{\tiny $\beta_{\hf}$}}
\put(14.7,-0.7){\makebox(0,0)[c]{\tiny $\beta_{\frac32}$}}
\put(19.6,-0.7){\makebox(0,0)[c]{\tiny $\beta_{n-\frac{1}{2}}$}}
\put(10.7,0.8){\tiny{$_{(1,-1)}$}}
\put(12.65,0.8){\tiny{$_{(-1,-1)}$}}
\put(14.7,0.8){\tiny{$_{(-1,-1)}$}}
\put(17.7,0.8){\tiny{$_{(-1,-1)}$}}
%\put(19.7,0.7){\tiny{$_{(-1,-1)}$}}
\end{picture}}
\end{equation}\vskip 1.3cm
\noindent
When $n=-1$ we mean the diagram \eqref{SG diagram} with no
$\beta_r$'s for $r\ge 1/2$. Similarly, let $\SG^\circ =\bigcup_n
\SG^\circ_n$ and let $\SG:=\SG^\circ\oplus\C d$ be its extension
by $d$ as in the case of $\DG^\circ$. Set $\SG_n:=\SG^\circ_n \oplus\C d$. Let
$\ov{\Pi}:=\Pi_B \cup \{\alpha_{-1}\} \cup \{\,\beta_j \mid j \in
\hf +\Z_+\, \}$ denote this fundamental system of $\SG$, and let
$\ov{\h}$ be the Cartan subalgebra of $\SG$.

\begin{prop}\label{prop:GSG:sub}
$\G$ and $\SG$ can be naturally regarded as subalgebras of $\DG$.
\end{prop}

\begin{proof}
We shall prove the case of $\G$, as the proof for $\SG$ is similar.
Consider the subdiagram:
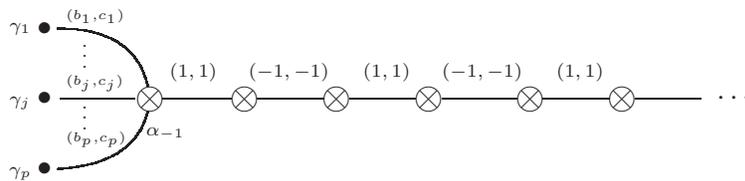
\begin{center}
\hskip 0cm \setlength{\unitlength}{0.23in}
\begin{picture}(22,3.3)
\put(2.85,2.6){\makebox(0,0)[c]{$\bullet$}}
\put(2.85,1){\makebox(0,0)[c]{$\bullet$}}
\put(2.85,-0.6){\makebox(0,0)[c]{$\bullet$}}
\put(5.25,1){\makebox(0,0)[c]{\Large $\otimes$}}
\put(7.4,1){\makebox(0,0)[c]{\Large $\otimes$}}
\put(11.6,1){\makebox(0,0)[c]{\Large $\otimes$}}
\put(9.5,1){\makebox(0,0)[c]{\Large $\otimes$}}
\put(13.9,1){\makebox(0,0)[c]{\Large $\otimes$}}
\put(16,1){\makebox(0,0)[c]{\Large $\otimes$}}
\put(18.6,1){\makebox(0,0)[c]{$\cdots$}}
\qbezier(5.2,1.3)(4.9,2.6)(3.15,2.6)
\qbezier(5.2,0.7)(4.9,-0.5)(3.15,-0.6)
\put(3.35,2.9){\tiny{$_{(b_1,c_1)}$}}
\put(3.7,2.45){\tiny{$_{\vdots}$}}
\put(3.35,1.4){\tiny{$_{(b_j,c_j)}$}}
\put(3.7,1.05){\tiny{$_{\vdots}$}}
\put(3.35,0.1){\tiny{$_{(b_p,c_p)}$}}
\put(3.2,1){\line(1,0){1.7}}
\put(5.55,1){\line(1,0){1.55}}
\put(7.7,1){\line(1,0){1.5}}
\put(9.8,1){\line(1,0){1.5}}
\put(11.9,1){\line(1,0){1.7}}
\put(14.2,1){\line(1,0){1.5}}
\put(16.3,1){\line(1,0){1.5}}
\put(5.7,1.5){\tiny{$(1,1)$}}
\put(7.5,1.5){\tiny{$(-1,-1)$}}
\put(10.1,1.5){\tiny{$(1,1)$}}
\put(11.9,1.5){\tiny{$(-1,-1)$}}
\put(14.5,1.5){\tiny{$(1,1)$}}
%\put(3.2,1.5){\tiny{$(b_j,c_j)$}}
%
%
\put(2.3,2.6){\makebox(0,0)[c]{\tiny $\gamma_{1}$}}
\put(2.3,0.9){\makebox(0,0)[c]{\tiny $\gamma_{j}$}}
\put(2.3,-0.7){\makebox(0,0)[c]{\tiny $\gamma_{p}$}}
\put(5.6,0.2){\makebox(0,0)[c]{\tiny $\alpha_{-1}$}}
\end{picture}
\end{center}\vskip 8mm
%Here {\Large$\RIGHTcircle$} means either $\bigcirc$ or \
%$\begin{picture}(2,2)\setlength{\unitlength}{0.14in}
%\put(0.3,0.3){\circle*{1}}
%\end{picture}$\ \, .
Applying the sequence of odd reflections with respect to
$\alpha_{\hf}$, $\alpha_{\frac{3}{2}}$, and
$\alpha_{\hf}+\alpha_{1}+\alpha_{\frac{3}{2}}$, we get the following
diagram by Proposition \ref{prop:odd:reflec}:
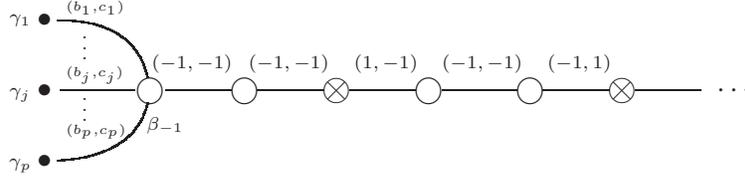
\begin{center}
\hskip 0cm \setlength{\unitlength}{0.23in}
\begin{picture}(22,3.3)
\put(2.85,2.6){\makebox(0,0)[c]{$\bullet$}}
\put(2.85,1){\makebox(0,0)[c]{$\bullet$}}
\put(2.85,-0.6){\makebox(0,0)[c]{$\bullet$}}
\put(5.25,1){\makebox(0,0)[c]{$\bigcirc$}}
\put(7.4,1){\makebox(0,0)[c]{$\bigcirc$}}
\put(11.6,1){\makebox(0,0)[c]{$\bigcirc$}}
\put(9.5,1){\makebox(0,0)[c]{\Large $\otimes$}}
\put(13.9,1){\makebox(0,0)[c]{$\bigcirc$}}
\put(16,1){\makebox(0,0)[c]{\Large $\otimes$}}
%\put(18.6,1){\makebox(0,0)[c]{\Large $\otimes$}}
\put(18.6,1){\makebox(0,0)[c]{$\cdots$}}
\put(3.2,1){\line(1,0){1.7}}
\put(5.6,1){\line(1,0){1.5}}
\put(7.7,1){\line(1,0){1.5}}
\put(9.8,1){\line(1,0){1.5}}
\put(11.9,1){\line(1,0){1.7}}
\put(14.2,1){\line(1,0){1.5}}
\put(16.3,1){\line(1,0){1.5}}
%\put(19,1){\line(1,0){1.7}}
\put(5.3,1.5){\tiny{$(-1,-1)$}}
\put(7.5,1.5){\tiny{$(-1,-1)$}}
\put(9.9,1.5){\tiny{$(1,-1)$}}
\put(11.9,1.5){\tiny{$(-1,-1)$}}
\put(14.3,1.5){\tiny{$(-1,1)$}}
\qbezier(5.2,1.3)(4.9,2.6)(3.15,2.6)
\qbezier(5.2,0.7)(4.9,-0.5)(3.15,-0.6)
\put(3.35,2.9){\tiny{$_{(b_1,c_1)}$}}
\put(3.7,2.45){\tiny{$_{\vdots}$}}
\put(3.35,1.4){\tiny{$_{(b_j,c_j)}$}}
\put(3.7,1.05){\tiny{$_{\vdots}$}}
\put(3.35,0.1){\tiny{$_{(b_p,c_p)}$}}
\put(2.3,2.6){\makebox(0,0)[c]{\tiny $\gamma_{1}$}}
\put(2.3,0.9){\makebox(0,0)[c]{\tiny $\gamma_{j}$}}
\put(2.3,-0.7){\makebox(0,0)[c]{\tiny $\gamma_{p}$}}
\put(5.6,0.2){\makebox(0,0)[c]{\tiny $\beta_{-1}$}}
\end{picture}
\end{center}\vskip 8mm
Note that in the process of {applying} odd reflections, we may have a matrix
with a diagonal entry $-2$. Then we replace it with $2$ by
multiplying a suitable diagonal matrix, which changes only the sign
of the corresponding row, but  does not change the associated
contragredient Lie superalgebra. Now, the subdiagram above, starting
from the first {isotropic} odd root, is of type $A$. Applying a suitable
sequence of odd reflections and using Proposition
\ref{prop:KM:embed}, it is easy to see that the contragredient Lie
superalgebra corresponding to any finite-rank diagram of the form
\eqref{G diagram} is a subalgebra of $\DG$. Finally we can check without difficulty that the outer derivation $d$ on $\DG^\circ$ when restricted to $\G^\circ$ coincides with $d$ on $\G^\circ$.  This completes the
proof.
\end{proof}

In light of Proposition~\ref{prop:GSG:sub}, we have
$$
\G_n=\DG_n\cap \G,\qquad \SG_m=\DG_m\cap \SG, \quad \text{ for }n\in
\N\text{ and }m\in\N\cup\{-1\}.
$$

\begin{cor}\label{root:mult}
Let $\DG$, $\SG$, and $\G$ be as above.
\begin{itemize}
\item[(1)]
If $\alpha$ is a root of $\G$, then $\alpha$ appears in both $\G$
and $\DG$ with the same multiplicity.

\item[(2)]
If $\alpha$ is a root of $\SG$, then $\alpha$ appears in both $\SG$
and $\DG$ with the same multiplicity.
\end{itemize}
\end{cor}

\begin{proof}
Without loss of generality, we may assume that $\alpha$ is positive.
So $\alpha$ is a finite sum of  simple roots of $\G$. By Proposition
\ref{prop:KM:embed}, it is a root of a contragredient Lie
superalgebra of finite rank, whose diagram is a connected subdiagram
of both $\G$ and $\DG$ (with respect to a fundamental system
different from its standard one, but constructed in the proof of
Proposition \ref{prop:GSG:sub}). This proves (1). The proof of (2)
is similar.
\end{proof}

We end this section with the following easy but important
observation on the conditions for $\G$ to be a (symmetrizable)
Kac-Moody Lie algebra. Recall that $c_j$ ($1\leq j\leq p$) are the
entries  in $A^{\texttt{hd}}$ satisfying \eqref{eq:pseudoCartan for
hd}.

\begin{prop}\label{crit:KM}  \mbox{}
\begin{itemize}
\item[(1)]
If  $\Pi_B$ has no odd simple root and $c_j\in-\Z_+$ for $1\leq
j\leq p$, then $\G$ is a Kac-Moody Lie algebra.

\item[(2)]
If in addition  $A^{\texttt{hd}}$ is symmetrizable, then the diagram
\eqref{G diagram} is also symmetrizable, and hence $\G$ is a
symmetrizable Kac-Moody Lie algebra.
\end{itemize}
\end{prop}

\begin{proof}
Since (1) is clear from the diagram of $\G$, we prove (2) only. Let
$D$ be an $l\times l$ diagonal   matrix such that $DA^{\texttt{hd}}$
is symmetric. Note that if $\underline{A}^{\texttt{hd}}$ is the
matrix obtained from $A^{\texttt{hd}}$ by replacing the diagonal
entry $0$ (for $\alpha_{-1}$) by $2$ and changing its parity, then
$D\underline{A}^{\texttt{hd}}$ is still symmetric. Let $d$ be the
diagonal entry of $D$ corresponding to $\alpha_{-1}$. Consider the
diagonal $(l +n-1)\times( l+n-1)$ matrix\vskip 2mm
\begin{equation*}
D':=\begin{pmatrix}
D&0&0&\cdots&0\\
0&d&0&\cdots&0\\
0&0&d&\cdots&0\\
0&\vdots&\vdots&\ddots&0\\
0&0&0&\cdots&d
\end{pmatrix}.
\end{equation*}\vskip 2mm

It follows that $D' A'$ is diagonal where $A'$ is the generalized
Cartan matrix corresponding to the  Dynkin diagram \eqref{G
diagram}.
\end{proof}

\begin{example}
If  $A^{\texttt{hd}}$ is one of the matrices corresponding to the
exceptional simple Lie superalgebras in Example~
\ref{ex:except:simples}, with $\alpha\in\N$, then the corresponding
contragredient Lie superalgebra $\G$ is a symmetrizable Kac-Moody
Lie algebra.
\end{example}

\section{Irreducible character formulas}\label{Irreducible characters}

Let $I$ be the index set for the SGCM $B$, {i.e.,
for} the vertices of its Dynkin diagram \framebox{$\mf{B}$}. Let
$I\cup\{-1\}$ be the index set for $A^{\texttt{hd}}$ or the head
diagram \eqref{Head diagram}, where as usual we denote the index of
the simple root $\alpha_{-1}$ by $-1$. Recall that $A^{\texttt{hd}}$
is again a SGCM by our assumption
\eqref{eq:pseudoCartan for hd}. The index set of the simple roots of
$\DG$ is $\widetilde{I}:=I\cup T_{\infty}$, where
$\wt{I}_{\ov{0}}=I_{\ov{0}}$ and $\wt{I}_{\ov{1}}=I_{\ov{1}}\cup
T_\infty$. The set of simple roots and coroots of $\DG$ are
$\wt{\Pi}=\{\,\alpha_i\,|\,i\in \wt{I}\,\}$ and
$\wt{\Pi}^\vee=\{\,\alpha^\vee_i\,|\,i\in \wt{I}\,\}$, respectively.
%We let $\wt{\Pi}^{\texttt{hd}}=\red{\Pi_B\cup\{\alpha_{-1}\}}$ be the set of simple roots in the head %diagram.

{From} now on, unless otherwise specified, we shall assume in addition
that
\begin{align}\label{eq:locally nilpotent}\tag{{\bf B}}
\begin{split}
\text{$B=(\langle \alpha_j,\alpha_i^\vee\rangle)_{i,j\in I}$ is a
symmetrizable and anisotropic  SGCM,}
%$\langle \alpha_i,\alpha_i^\vee\rangle=2$ for all $i\in I$},
\end{split}
\end{align}
that is, $\G(B)$ is a symmetrizable anisotropic Kac-Moody
superalgebra, whose structure and representations were studied in
detail in \cite{K2} (also cf. \cite[Remark 4.11]{CFLW}).
In particular, ${\rm ad}(e_i)$ and ${\rm
ad}(f_i)$ act locally nilpotently on $\DG$ for $i\in {I}$ by Remark~
\ref{rem:Serre relations}. Note that the off-diagonal entries $c_j$
in $A^{\texttt{hd}}$ may still be positive integers.

For $i\in \widetilde{I}$, let $\omega_i$  be the fundamental weight
for $\DG$ such that
\begin{align*}
\langle\omega_i,\alpha^\vee_j\rangle=\delta_{ij},\quad \text{for $i,j\in \widetilde{I}$}.
\end{align*}
We further assume that $\langle \omega_i,d \rangle=-1$ for $i \in
T_\infty$, and $0$ for $i\in I$. Let $\omega$ be the fundamental
weight with respect to $d$, that is, $\langle\omega,d\rangle=1$ and
$\langle\omega,\alpha^\vee_i\rangle=0$ for all $i\in \wt{I}$. For
$i\in T_\infty$, define
\begin{equation}\label{epsilon and h}
\begin{split}
\epsilon_i&:=
\begin{cases}
\omega_{-1}, & \text{if $i=-1$},\\
-\omega_{\hf} +\omega_{-1}, & \text{if $i=\hf$},\\
(-1)^{2i}\left(\omega_i-\omega_{i-\hf}\right), & \text{if $i> \hf$},
\end{cases}\\
h_i&:=
\begin{cases}
-d, & \text{if $i=-1$},\\
d+\alpha_{-1}^\vee, & \text{if $i=\hf$},\\
-h_{i-\hf}-(-1)^{2i}\alpha^\vee_{i-\hf}, & \text{if $i>\hf$}.
\end{cases}
\end{split}
\end{equation}
Then $\langle \epsilon_i,h_j \rangle=\delta_{ij}$ and $\langle
\omega_k,h_j\rangle =0$ for $i,j\in T_\infty$ and $k\in I$.

\begin{rem}\label{rem:gl}
Let $\gl(\infty|\infty)$ denote the general linear  Lie superalgebra
spanned by the elementary matrices $E_{i,j}$  for $i,j\in T_\infty$
with parity $p(E_{i,j})=2(i+j) \pmod{2}$. Then the subalgebra of
$\DG$ generated by $e_i$, $f_i$, and $h_i$ for $i\in T_\infty$  is
isomorphic to $\gl(\infty|\infty)$, where $e_{-1}\mapsto
E_{-1,1/2}$, $e_r\mapsto E_{r, r+1/2}$, $f_{-1}\mapsto E_{1/2, -1}$,
$f_r\mapsto (-1)^{2r}E_{r+1/2, r}$ for $r\in \hf\N$, and $h_i\mapsto E_{i,i}$
for $i\in T_\infty$.
\end{rem}

Let $\wt{P}=\Z\omega+ \sum_{i\in \wt{I}}\Z\omega_i $ be the set of
integral weights for $\DG$. Note
that $\alpha_r=\epsilon_r-\epsilon_{r+1/2}\in \wt{P}$ for $r\in\hf\N$. We define
\begin{align*}
&P:=\Z\omega+\sum_{i\in I\cup\{-1\}}\Z\,\omega_i+\sum_{n\in\N}\Z\epsilon_n,\ \ \ \
\ov{P}:=\Z\omega+\sum_{i\in I\cup\{-1\}}\Z\,\omega_i+\sum_{r\in\hf+\Z_+}\Z\epsilon_r.
\end{align*}
We can regard $P\subseteq \wt{P}$ and $\ov{P}\subseteq\wt{P}$. Introduce
$$\varpi_{-1}=\omega_{-1},
\quad \varpi_{1}=\epsilon_1+\omega_{-1}, \quad
\varpi_{n}=\epsilon_n+\varpi_{n-1} \; (n\ge 2),
$$
$$
\varpi_{1/2}=\omega_{1/2}-\omega_{-1}=-\epsilon_{1/2},
\quad
\varpi_{r}=-\epsilon_{r} +\varpi_{r-1}\; (r\in\hf+\N).
$$
Then we
have
\begin{align*}
\langle \varpi_i,\beta^\vee_{j}\rangle=\delta_{ij}, \ \  \text{ for
} i,j\in \{-1\}\cup \tfrac{1}{2}\mathbb{N}.
\end{align*}
In particular, the sets $\{\varpi_{-1},\varpi_1,\varpi_2,\ldots\}$
and
$\{\varpi_{\hf},\varpi_{\frac{3}{2}},\varpi_{\frac{5}{2}},\ldots\}$
are the fundamental weights for the (even) tail diagrams of $\G$ and
$\SG$, starting from $\beta_{-1}$ and $\beta_{1/2}$, respectively.
Thus, we can identify $P$ and $\ov{P}$ with the sets of integral
weights for $\G$ and $\SG$, respectively.

Now, we define  functors $T$ and $\ov{T}$ as follows. For a
$\DG$-module $\wt{M}$ with weight space decomposition
$\wt{M}=\bigoplus_{\gamma\in \wt{P}}\wt{M}_{\gamma}$, we set

\begin{equation*}
T(\wt{M}):=\bigoplus_{\gamma\in P}\wt{M}_\gamma,\ \ \ \ \
\ov{T}(\wt{M}):=\bigoplus_{\gamma\in \ov{P}}\wt{M}_\gamma.
\end{equation*}
A homomorphism $f:\wt{M}\rightarrow\wt{N}$ of $\DG$-modules with
integral weights is in particular an $\wt{\h}$-module homomorphism.
Hence we have the restriction maps $T[f]:=f|_{T(\wt{M})}:  T(\wt{M})
\longrightarrow T(\wt{N})$ and $\ov{T}[f]:= f|_{\ov{T}(\wt{M})} :
\ov{T}(\wt{M})\longrightarrow \ov{T}(\wt{N})$, respectively.

\begin{prop}\label{Functor T and G-modules}
$T(\wt{M})$ and $\ov{T}(\wt{M})$ are $\G$- and $\SG$-modules, respectively.
\end{prop}

\begin{proof}
We will show this for $T(\wt{M})$ only, as the case of
$\ov{T}(\wt{M})$ is similar. Writing the simple roots
$\beta_n=\alpha_{n}+\alpha_{n+1/2}$ ($n\in\N$) in terms of
fundamental weights, we see that
$\beta_n=(\omega_n-\omega_{n-1/2})+(\omega_{n+1/2}-\omega_{n+1})
=\epsilon_n-\epsilon_{n+1}$. Also we have
\begin{equation*}
\beta_{-1}=\alpha_{-1}+\alpha_{\hf}=\sum_{i\in
I\cup\{-1\}}\kappa_i\omega_i +\omega_{\hf}-\omega_1 =\sum_{i\in
I\cup\{-1\}}\kappa_i\omega_i -\epsilon_1 \quad(\kappa_i\in\Z).
\end{equation*}
Hence, $P$ is invariant under adding or subtracting the simple roots
of $\G$, which implies that $T(\wt{M})$ is a $\G$-module.
\end{proof}

Let $J\subseteq I$ be given. Let $\wt{\mf l}^{\, \texttt{hd}}$ be
the Levi subalgebra of $\DG$ generated by the Cartan subalgebra of
$\G(A^{\texttt{hd}})$ and $\{\,e_{i}, f_i\,|\, i\in J\,\}$, which is
a symmetrizable anisotropic Kac-Moody Lie superalgebra by
\eqref{eq:locally nilpotent}. Furthermore, let
$\gl(\infty|\infty)_{>0}$ be the subalgebra of $\gl(\infty|\infty)$
spanned by $E_{r,s}$ for $r,s\in T_\infty\setminus\{-1\}$  (see
Remark \ref{rem:gl}). Put
$\wt{\mf{l}}=\wt{\mf{l}}^{\,\texttt{hd}}\oplus
\gl(\infty|\infty)_{>0}$.

Let $\DG=\wt{\mf n}^+\oplus \wt{\mf h}\oplus \wt{\mf n}^-$ be the
triangular decomposition of $\DG$, and let $\wt{\mf p}=\wt{\mf
l}+\wt{\mf n}^+$ be the parabolic subalgebra corresponding to
$\wt{\mf l}$ with nilradical $\wt{\mf{u}}^+$ and opposite radical
$\wt{\mf{u}}^-$. Let $\ov{\mf l}=\SG\cap \wt{\mf l}$ and ${\mf
l}=\G\cap \wt{\mf l}$. The subalgebras $\ov{\mf{u}}^\pm$, $\ov{\mf
p}$ of $\SG$, and ${\mf{u}}^\pm$, ${\mf p}$ of $\G$ are defined in a
similar way.

Denote the set of partitions by $\mc P$. For
$\mu=(\mu_1,\mu_2,\ldots)\in \mc P$, let
$\mu'=(\mu'_1,\mu'_2,\ldots)$ denote the conjugate of $\mu$. We
denote a variant of Frobenius coordinates for $\mu$ by
\begin{equation*}
\mu^\theta=(\mu^\theta_{\hf},\mu^\theta_1,\mu^\theta_{\frac{3}{2}},\ldots)
:=(\mu'_1,\langle\mu_1-1\rangle,\langle\mu'_2-1\rangle,\langle\mu_2-2\rangle,\ldots),
\end{equation*}
where $\langle a\rangle=\max\{a,0\}$ for $a\in\Z$.

Let $P^+$ be the subset of $P$ consisting of weights of the following form:
\begin{align}\label{lambda:weight}
\la=\kappa\omega+\sum_{i\in I\cup\{-1\}} \kappa_i\omega_i
+\sum_{n\in\N}{}^+\la_n\epsilon_n,
\end{align}
where ${}^+\la=({}^+\la_1,{}^+\la_2,\ldots)\in\mc P$,
and $\kappa_i\in\Z_+$ (respectively $\kappa_i\in2\Z_+$) for $i\in
J_{\ov{0}}$ (respectively $i\in J_{\ov{1}}$).
%We shall write ${}^+\la$ for the partition $\mu$ in \eqref{lambda:weight}.
For $\la\in P^+$ as in \eqref{lambda:weight}, we define
\begin{equation}\label{natural and theta}
\begin{split}
&\la^\natural:=\kappa\omega+\sum_{i\in I\cup\{-1\}}\kappa_i\omega_i
+\sum_{r\in \hf+\Z_+}{}^+\la'_{r+\hf}\epsilon_r\in \ov{P},
 \\
&\la^\theta:=\kappa\omega+\sum_{i\in I\cup\{-1\}}\kappa_i\omega_i
+\sum_{r\in \hf\N}{}^+\la^\theta_r\epsilon_{r} \in\wt{{P}}.
\end{split}
\end{equation}
We denote the images of $\natural: P^+\rightarrow \ov{P}$ and
$\theta: P^+\rightarrow \wt{P}$ by $\ov{P}^+$ and $\wt{P}^+$,
respectively, so that we get respective bijections
$\natural:P^+\rightarrow \ov{P}^+$ and $\theta:P^+\rightarrow
\wt{P}^+$.

For $\la\in P^+$, we assume the following notations.

\begin{itemize}
\item[]
$L(\mf{l},\la)$ : the irreducible $\mf{l}$-module with highest
weight $\la$,

\item[]
$\Delta(\la)$ : the parabolic Verma module ${\rm Ind}_{\mf
p}^{\G}L(\mf{l},\la)$ with highest weight $\la$,

\item[] $L(\la)$ : the unique irreducible quotient of $\Delta(\la)$,
\end{itemize}
where we extend $L(\mf{l},\la)$ in a trivial way to a
$\mf{p}$-module. Similarly as above, we introduce the
self-explanatory notations of $L(\ov{\mf{l}},\la^\natural)$,
$\ov{\Delta}(\la^\natural)$, $\ov{L}(\la^\natural)$, and
$L(\wt{\mf{l}},\la^\theta)$, $\wt{\Delta}(\la^\theta)$,
$\wt{L}(\la^\theta)$ for $\la\in P^+$.

Now, we define ${\mc O}$ to be the category of $\G$-modules $M$
such that $M$ is ${\h}$-semisimple with finite-dimensional weight
spaces and satisfies
\begin{itemize}
\item[(1)]
$M$ is a direct sum of $L(\mf{l},\gamma)$'s with $\gamma\in P^+$ as
an $\mf{l}$-module,

\item[(2)]
there exist  $\la^1,\la^2,\ldots,\la^k\in P^+$ such that ${\rm
wt}(M)\subset
\bigcup_{i}\left(\la^i-\sum_{\alpha\in\Pi}\Z_+\alpha\right)$,
\end{itemize}
where ${\rm wt}(M)$ denotes the set of weights of $M$. The morphisms in
$\mc O$ are all $\G$-module homomorphisms. The categories $\ov{\mc O}$ and $\wt{\mc
O}$ of $\SG$- and $\DG$-modules,
respectively, are defined analogously, with $\h$, $\mf{l}$, $P^+$,
and $\Pi$ replaced accordingly.

\begin{lem} \label{lem:DeltaL}
For  $\la\in P^+$, we have $X(\la)\in\mc O$,
$\ov{X}(\la^\natural)\in\ov{\mc O}$, and
$\wt{X}(\la^\theta)\in\wt{\mc O}$, where $X=\Delta$ or $L$.
\end{lem}

\begin{proof}
First, we note that both $L(\mf{l},\la)$ and $\mf{u}^-$ are
integrable $\mf{l}$-modules. The enveloping algebra $U(\mf{u}^-)$ is
also integrable via the adjoint action of $\mf l$. Hence,
$U(\mf{u}^-)\otimes L(\mf{l},\la)$, which is isomorphic to
$\Delta(\la)$ as an $\mf{l}$-module, is completely reducible over
$\mf{l}$ \cite[Proposition 2.8]{K2} (see also \cite[Section 3.2.2]{CK}).  This proves $\Delta(\la)\in
\mc O$ and hence $L(\la)\in \mc O$.

Next, we see that $L(\ov{\mf{l}},\la^\natural)$ is an integrable
$\ov{\mf l}$-module. Also, by \eqref{eq:pseudoCartan for hd}, \eqref{eq:locally nilpotent},
 and Remark \ref{rem:Serre
relations}, $\ov{\mf{u}}^-$ is an integrable $\ov{\mf{l}}$-module,
and so is its enveloping algebra.  This shows
$\ov{X}(\la^\natural)\in\ov{\mc O}$.

It remains to show $\wt{X}(\la^\theta)\in\wt{\mc O}$. Since $\wt{\mf
l}=\wt{\mf l}^{\, \texttt{hd}}\oplus\gl(\infty|\infty)_{>0}$,
$L(\wt{\mf{l}},\la^\theta)$ is an integrable $\wt{\mf l}^{\,
\texttt{hd}}$-module and a polynomial
$\gl(\infty|\infty)_{>0}$-module. It is clear from \eqref{eq:pseudoCartan for hd}, \eqref{eq:locally
nilpotent},  and Remark
\ref{rem:Serre relations} that $\wt{\mf{u}}^-$ is an integrable
$\wt{\mf l}^{\, \texttt{hd}}$-module.

Now, we claim that $\wt{\mf{u}}^-$ is a polynomial
$\gl(\infty|\infty)_{>0}$-module. Let $\alpha$ be a negative root in
$\wt{\mf{u}}^-$, where $\alpha=-\sum_{i\in \wt{I}}a_i\alpha_i$ for
some $a_i\in \Z_+$ with $a_{-1}\neq 0$. Since
$\alpha_{-1}=\omega_{1/2}+\sum_{i\in I\cup\{-1\}}\langle
\alpha_{-1},\alpha^\vee_i \rangle\omega_i$ and
$\omega_{1/2}=-\epsilon_{1/2}+\omega_{-1}$,
\begin{equation}\label{root of u}
\alpha=\sum_{i\in I\cup\{-1\}}p_i \omega_i +\sum_{r\in \hf\N}q_r \epsilon_r,
\end{equation}
for some $p_i$, $q_r\in \Z$. Since $a_{-1}\neq 0$, there exists an
$s\in\hf\N$ such that $q_s\neq 0$. Choose the largest such $s$. Then
$q_s\in \N$ since $\alpha_{s-1/2}=\epsilon_{s-1/2}-\epsilon_s$ and
$a_{s-1/2}\in \N$. Suppose that $q_t$ is negative for some $1/2\leq
t<s$, and choose the largest such $t$. Let $x_\pm$ be a root vector
of $\DG$ associated with an even root
$\pm\beta_{t}=\pm(\alpha_t+\alpha_{t+1/2})$. Then ${\rm ad}(x_\pm)$
is locally nilpotent on $\wt{\mf u}^-$. We apply the simple
reflection associated with $\beta_t$ to have another negative root
$\alpha'$ of $\wt{\mf u}^-$, which is obtained from $\alpha$ by
exchanging  the coefficients of $\epsilon_t$ and $\epsilon_{t+1}$.
Applying the above process repeatedly, we have a negative root of
the form \eqref{root of u}, where there exists $s\in\hf\N$ such that
$q_s\in -\N$ and $q_{s'}=0$ for $s'>s$, which  is a contradiction.
This proves our claim, which implies that the enveloping algebra
$U(\wt{\mf u}^-)$ is also a polynomial
$\gl(\infty|\infty)_{>0}$-module.

Therefore, $\wt{\Delta}(\la^\theta)$ is completely reducible over
$\wt{\mf l}$, since $U(\wt{\mf u}^-)\otimes
L(\wt{\mf{l}},\la^\theta)$ is an integrable $\wt{\mf l}^{\,
\texttt{hd}}$-module and a polynomial
$\gl(\infty|\infty)_{>0}$-module (see e.g.~ \cite{CK}, \cite[Theorem
6.4]{CW}). We conclude that $\wt{\Delta}(\la^\theta)\in\wt{\mc O}$
and hence $\wt{L}(\la^\theta)\in\wt{\mc O}$.
\end{proof}

Let $n\in\mathbb{N}$ be given. Consider the following sequence of
$\frac{n(n+1)}{2}$ odd roots in $\DG$.
\begin{equation*}
\begin{split}
\alpha_{\hf},\ \ &
 \underbrace{\alpha_{\frac{3}{2}},\  \alpha_{\frac{3}{2}}+\alpha_1+\alpha_\hf}_{2},\
 \underbrace{\alpha_{\frac{5}{2}},\
  \alpha_{\frac{5}{2}}+\alpha_2+\alpha_{\frac{3}{2}},
 \ \alpha_{\frac{5}{2}}+\alpha_2
 +\alpha_{\frac{3}{2}}+\alpha_1+\alpha_{\hf}}_{3},\ \ldots \\
&\ \ \ \ \ \ldots \ ,
\underbrace{\alpha_{\frac{2n-1}{2}},\ \alpha_{\frac{2n-1}{2}}
+\alpha_{\frac{2n-2}{2}}+\alpha_{\frac{2n-3}{2}},\ \ldots,
\ \alpha_{\frac{2n-1}{2}}+\alpha_{\frac{2n-2}{2}}+\cdots+\alpha_{\hf}}_{n}. \\
%&  \underbrace{\alpha_{\frac{5}{2}}, \ \alpha_{\frac{3}{2}}+\alpha_{2}+\alpha_{\frac{5}{2}},\
% \alpha_{\hf}+\alpha_1+\alpha_{\frac{3}{2}}+\alpha_{2}+\alpha_{\frac{5}{2}}}_{3}
\end{split}
\end{equation*}
We apply successively the corresponding odd reflections to $\DG$ in
the above order starting from $\alpha_\hf$. The resulting new Borel
subalgebra of $\DG$ is denoted by $\wt{\mf{b}}^c(n)$,  whose set of
simple roots is as follows:
${\Pi}_B\cup\{\beta_{-1},\beta_1,\cdots,\beta_{n-1}\}
\cup\{-\sum_{k=1}^{2n-1}\alpha_{\frac{k}{2}}\}
\cup\{\beta_{\hf},\beta_{\frac{3}{2}},\cdots, \beta_{n-\hf}\}
\cup\{\alpha_{n+\hf},\alpha_{n+1},\alpha_{n+\frac{3}{2}},\cdots\}$.

Next, consider another sequence of $\frac{n(n+1)}{2}$ odd roots in
$\DG$ as follows.
\begin{equation*}
\begin{split}
\alpha_{1},\ \ &
 \underbrace{\alpha_{2},\  \alpha_2+\alpha_{\frac{3}{2}}+\alpha_1}_{2},\
 \underbrace{\alpha_{3},\  \alpha_3+\alpha_{\frac{5}{2}}+\alpha_2 ,\
 \alpha_3+\alpha_{\frac{5}{2}}+\alpha_2+\alpha_{\frac{3}{2}}+\alpha_1}_{3},\ \ldots
  \\
&\ \ \  \ \ \ \  \ \ \ \ \ \ \ \ldots \ ,
\underbrace{\alpha_{n},\ \alpha_{n}+\alpha_{\frac{2n-1}{2}}
+\alpha_{\frac{2n-2}{2}},\ \ldots,\ \alpha_{n}+\alpha_{\frac{2n-1}{2}}
+\cdots+\alpha_{1}}_{n}.
 \\
%&  \underbrace{\alpha_{\frac{5}{2}}, \ \alpha_{\frac{3}{2}}+\alpha_{2}+\alpha_{\frac{5}{2}},\
% \alpha_{\hf}+\alpha_1+\alpha_{\frac{3}{2}}+\alpha_{2}+\alpha_{\frac{5}{2}}}_{3}
\end{split}
\end{equation*}
We denote by $\wt{\mf{b}}^s(n)$ the new Borel subalgebra of $\DG$
obtained by applying the odd reflections corresponding to the above
sequence starting from $\alpha_1$. The set of simple roots of $\wt{\mf{b}}^s(n)$ is
$\Pi_B\cup\{\alpha_{-1},\beta_{\hf},\beta_{\frac{3}{2}},\cdots,\beta_{n-\hf}
\}\cup\{-\sum_{k=2}^{2n}\alpha_{\frac{k}{2}}\}
\cup\{\beta_{1},\beta_{2},\cdots,\beta_{n}\}
\cup\{\alpha_{n+1},\alpha_{n+\frac{3}{2}},\alpha_{n+2},\cdots\}$.

The proof of  \cite[Lemma 3.1]{CL} can be easily adapted (as done in
\cite{CLW}) to prove the following.

\begin{lem}\label{lem:change}
Let $\la\in{P}^+$ and $n\in\N$ be given.
\begin{itemize}
\item[(1)]
If $\ell({}^+\la)\le n$, then the highest weight of
$L(\wt{\mf{l}},\la^\theta)$ with respect to $\wt{\mf{b}}^c(n)$ is
$\la$.

\item[(2)]
If $\ell(({}^+\la)')\le n$, then the highest weight of
$L(\wt{\mf{l}},\la^\theta)$ with respect to  $\wt{\mf{b}}^s(n)$ is
$\la^\natural$.
\end{itemize}
Here $\ell(\mu)$ denotes the length of a partition $\mu$.
\end{lem}

\begin{lem}\label{lem:chL}
For $\la\in{{P}}^+$, we have $T({L}(\wt{\mf
l},\la^\theta))=L(\mf{l},\la)$ and  $\ov{T}({L}(\wt{\mf
l},\la^\theta))={L}(\ov{\mf l},\la^\natural)$.
\end{lem}

\begin{proof}
Let $\{\,x_1,x_2,\ldots\,\}$ and $\{\,y_1,y_2,\ldots\,\}$ be
mutually commuting formal variables. For $\mu\in \mc{P}$, let
$s_\mu(x_1,x_2,\ldots)$ and $HS_\mu(x_1,x_2,\ldots ;y_1,y_2,\ldots)$
denote the Schur function and the hook (or super) Schur function
corresponding to $\mu$, respectively (cf. \cite[Appendix~A]{CW}).
For an indeterminate $e$,  set $x_{i}=e^{\epsilon_i}$ for
$i\in\hf\N$. Let $\la\in P^+$ be as in \eqref{lambda:weight}.  Since
$\wt{\mf l}=\wt{\mf l}^{\,
\texttt{hd}}\oplus\gl(\infty|\infty)_{>0}$, we see that (ignoring
the eigenvalue of $d$)
\begin{align*}
\text{ch\,} L(\wt{\mf{l}},\la^\theta)=\text{ch\,}L(\wt{\mf l}^{\,
\texttt{hd}},\nu)HS_{{}^+\la}(x_1,x_2,\cdots;x_{1/2},x_{3/2},\cdots),
\end{align*}
where $\nu=\sum_{i\in I }\kappa_i\omega_i$ (cf. \cite{CK}) and
$L(\wt{\mf l}^{\, \texttt{hd}},\nu)$ is the irreducible highest weight $\wt{\mf
l}^{\, \texttt{hd}}$-module with highest weight $\nu$. In
particular, it follows that
\begin{align*}
\text{ch\,} T(L(\wt{\mf{l}},\la^\theta))&=\text{ch\,}
L(\wt{\mf l}^{\, \texttt{hd}},\nu)s_{{}^+\la}(x_1,x_2,\cdots),\\
\text{ch\,} \ov{T}(L(\wt{\mf{l}},\la^\theta))&=\text{ch\,}
L(\wt{\mf l}^{\, \texttt{hd}},\nu)s_{({}^+\la)'}(x_{1/2},x_{3/2},\cdots).
\end{align*}
Then the lemma follows from comparing the above characters and the
definitions of $T$ and $\ov{T}$.
\end{proof}

\begin{prop}\label{prop:T is exact}
$T$ and $\ov{T}$ are exact functors from $\wt{\mc{O}}$ to $\mc O$
and $\ov{\mc O}$, respectively.
\end{prop}

\begin{proof}
By Lemma \ref{lem:chL}, it suffices to show that $T(\wt{M})\in {\mc
O}$ and $\ov{T}(\wt{M})\in \ov{\mc O}$ for $\wt{M}\in \wt{\mc O}$.
This follows from the same argument as in \cite[Proposition
6.15]{CW}.
\end{proof}

\begin{lem}\label{borel2borel}
We have $T(\wt{\mf{u}}^+)={\mf{u}}^+$ and \
$\ov{T}(\wt{\mf{u}}^+)=\ov{\mf{u}}^+$.
\end{lem}

\begin{proof}
We show  $T(\wt{\mf{u}}^+)={\mf{u}}^+$ only, as the proof of
$\ov{T}(\wt{\mf{u}}^+)=\ov{\mf{u}}^+$ is similar.

It is clear that the simple roots  $\beta_{-1}$ and $\beta_n$
($n\in\N$) lie in $T(\wt{\mf{u}}^+)$ (cf. Proposition \ref{Functor T
and G-modules}). Thus ${\mf{u}^+}\subseteq T(\wt{\mf{u}}^+)$.
Conversely, let $\alpha$ be a positive  root of $\DG$ such that
$\alpha\in P$. Then $\alpha$ is of the form $\sum_{i\in
I}a_i\alpha_i+a_{-1}\beta_{-1}+\sum_{n\in\N}b_n\beta_n$ for some
$a_i,b_n\in\Z_+$. So it remains to show that $\alpha$ is generated
by the root vectors corresponding to the simple roots $\alpha_i$,
$\beta_{-1}$, and $\beta_n$ for $i\in I$ and $n\in\N$. This follows
from the procedure of applying sequences of odd reflections to
obtain $\wt{\mf{b}}^c(n)$ from $\wt{\mf{b}}$. Now by Corollary
\ref{root:mult}, the root multiplicities also coincide. Hence,
$T(\wt{\mf{u}}^+)\subseteq {\mf{u}^+}$.
\end{proof}

\begin{prop}\label{prop:Verma to Verma}
For $\la\in{{P}}^+$, we have $T(\wt{X}(\la^\theta))=X(\la)$ and
$\ov{T}(\wt{X}(\la^\theta))=\ov{X}(\la^\natural)$, where $X=\Delta$
or $L$.
\end{prop}

\begin{proof}
We shall sketch the proof for $T$, as the case for $\ov{T}$ is similar.

With Lemmas \ref{lem:change} and \ref{lem:chL} at our disposal, we
may follow the proof of \cite[Proposition 6.16]{CW} and show that
$T(\wt{X}(\la^\theta))$ is a highest weight module in $\mc O$ with
highest weight $\la$.  Now Lemma \ref{borel2borel} shows that
$T(\wt{\Delta}(\la^\theta))$ has the same character as
$\Delta(\la)$, and hence it must be equal to $\Delta(\la)$. The
irreducibility of $T(\wt{L}(\la^\theta))$ follows from the same
argument as in \cite[Theorem 6.17]{CW}, and hence
$T(\wt{L}(\la^\theta))=L(\la)$.
\end{proof}

For $\la\in P^+$, we write that
\begin{equation} \label{eq:charKL}
\text{ch\,} L(\la)=\sum_{\mu\in P^+}m_{\mu\la}\text{ch\,} \Delta(\mu)
\end{equation}
with $m_{\mu\la}\in\Z$. By Propositions~ \ref{prop:T is exact} and \ref{prop:Verma to
Verma}, we conclude the following.

\begin{thm}\label{thm:irred:char}
The character formulas of the irreducible $\DG$- and $\SG$-modules
in the respective categories $\wt{\mc O}$ and $\ov{\mc O}$ are
determined by those of the irreducible $\G$-modules in $\mc{O}$.
More precisely, for  $\la \in P^+$,  the characters are given by
\begin{align*}
\text{ch\,} \wt{L}(\la^\theta) &=\sum_{\mu\in
P^+}m_{\mu\la}\text{ch\,} \wt{\Delta}(\mu^\theta),
 \\
\text{ch\,} \ov{L}(\la^\natural)&=\sum_{\mu\in
P^+}m_{\mu\la}\text{ch\,} \ov{\Delta}(\mu^\natural),
\end{align*}
where $m_{\mu\la}$ are as in \eqref{eq:charKL}.
\end{thm}\vskip 2mm

Finally, let us recall the truncation functors to describe the
irreducible characters for the corresponding Lie superalgebras of
finite rank.

For $n\in\N\cup\{-1\}$, recall the finite-rank Lie (super)algebras
$\DG_n$,  $\SG_n$, and $\G_n$ from Section~\ref{Three Lie
superalgebras}, whenever they are defined. Let $\wt{P}_n$ be the set
of $\la\in \wt{P}$ such that the coefficient of $\epsilon_r$ in
$\la$ is $0$ for $r>n+\hf$. We may regard $\la\in \wt{P}_n$ as an
integral weight for $\DG_n$. Note that
$\omega=-\epsilon_{-1}+\epsilon_{1/2}-\epsilon_1+\cdots -\epsilon_n+\epsilon_{n+\hf}$
as a weight for $\DG_n$. Similarly we let $\ov{P}_n=\ov{P}\cap
\wt{P}_n$ and ${P}_n={P}\cap \wt{P}_n$ be the sets of integral
weights for $\SG_n$ and $\G_n$, respectively. Put
$\wt{P}^+_n=\wt{P}_n\cap \wt{P}^+$, $\ov{P}^+_n=\ov{P}_n\cap
\ov{P}^+$, and ${P}^+_n={P}_n\cap {P}^+$.

For $\la\in P^+_n$, we let $\Delta_n(\lambda)$ and $L_n(\la)$ be the
corresponding parabolic Verma and irreducible modules over $\G_n$
with highest weight $\la$, and let $\mc{O}_n$ denote the
corresponding category of $\G_n$-modules. Similarly,
$\ov{\Delta}_n(\la^\natural)$, $\ov{L}_n(\la^\natural)$, $\ov{\mc
O}_n$ and $\wt{\Delta}_n(\la^\theta)$, $\wt{L}_n(\la^\theta)$,
$\wt{\mc O}_n$ are defined for $\SG_n$ and $\DG_n$, respectively.

For $n< k\leq \infty$, we define the truncation functor $\mf{tr}^k_n
: \mc{O}_k \longrightarrow \mc{O}_n$ by sending $M=\bigoplus_{\gamma
\in P_k}M_\gamma$ to its subspace $\bigoplus_{\gamma \in
P_n}M_\gamma$. The truncation functors from $\ov{\mc O}_k$ and
$\wt{\mc O}_k$ to $\ov{\mc O}_n$ and  $\wt{\mc O}_n$, respectively,
are defined analogously. Here it is understood that
$\mc{O}_\infty=\mc{O}$, $\ov{\mc O}_\infty=\ov{\mc O}$, and $\wt{\mc
O}_\infty=\wt{\mc O}$. Then for  $\mu\in P^+_k$, we have
\begin{equation}\label{truncation}
\mf{tr}^k_n(X_k(\mu))=
\begin{cases}
X_n(\mu), & \text{if $\mu\in P^+_n$},\\
0, & \text{otherwise},
\end{cases}
\end{equation}
where $X=\Delta$ or $L$ (cf. \cite[Proposition 6.9]{CW}). The
counterparts of \eqref{truncation} hold for $\mf{tr}^k_n: \ov{\mc
O}_k \rightarrow \ov{\mc O}_n$ (respectively $\mf{tr}^k_n: \wt{\mc O}_k
\rightarrow \wt{\mc O}_n$) with $P^+_k$ and $X_k(\mu)$  replaced by
$\ov{P}^+_k$ and $\ov{X}_k(\mu)$ (respectively $\wt{P}^+_k$ and
$\wt{X}_k(\mu)$). Now Theorem~\ref{thm:irred:char}  has the
following finite-rank analogue thanks to the counterpart for
$\ov{\mc O}_n$ of \eqref{truncation}.

\begin{cor}\label{cor:KL}
For $\la \in P^+$ such that $\la^\natural \in \ov{P}^+_n$,   the
character formula of the irreducible $\SG_n$-module $ \ov{L}_n
(\la^\natural)$ is given by
\begin{align*}
\text{ch\,} \ov{L}_n (\la^\natural)&=\sum_{\{\mu\in P^+| \mu^\natural \in \ov{P}^+_n\}}
m_{\mu\la}\text{ch\,} \ov{\Delta}_n (\mu^\natural),
% \\
%\text{ch\,} \wt{L}(\la^\theta) &=\sum_{\mu\in P^+}m_{\mu\la}\text{ch\,} \wt{\Delta}(\mu^\theta).
\end{align*}
where $m_{\mu\la}$ are as in \eqref{eq:charKL}.
\end{cor}

\begin{rem}  \label{rem:B}
(1) Assume that $I_{\ov{1}}=\emptyset$ (and thus $B$ is a GCM),
$\langle\gamma_j,\alpha^\vee_{-1}\rangle=c_j\le 0$ for all $1\leq
j\leq p$, and $A^{\texttt{hd}}$ is symmetrizable. Then the integers
$m_{\mu\la}$ are often the evaluation at $q=1$ of
various parabolic Kazhdan-Lusztig polynomials for the Hecke algebra
associated to the Weyl group of a symmetrizable Kac-Moody Lie
algebra $\G$ (see Kashiwara-Tanisaki
\cite{KaTa2} and the references therein).

(2) Assume that $B$ is an anisotropic SGCM
satisfying additionally that the
square length of $\alpha_i$ has the same parity as $i$, for all
$i\in I$ (this additional condition might be removable eventually).
For such a $B$, an additional equivalence of categories
between $\mc O_n$ and a BGG category of some Kac-Moody Lie algebra
given in \cite[Remark~4.11]{CFLW}, when combined with
Corollary~\ref{cor:KL}, still allows us to determine the characters
of irreducible $\SG_n$-modules in the category $\ov{\mc O}_n$ in
terms of those of the corresponding Kac-Moody Lie algebra.

(3) Recall that $\G=\G^\circ+\C d$.  We note that in the categories that we consider in this paper the multiplicities $m_{\mu\la}$ in \eqref{eq:charKL} for $\G$ and $\G^\circ$ are identical. This can be seen as follows: First, note that the respective irreducible $\mf l$- and $\mf l^\circ$-modules are the same using the exact same type of argument as in Exercise 2.1 in \cite{CW}. From this it follows that the parabolic $\G$- and $\G^\circ$-modules are also the same. Now, since the irreducible $\G$-modules and $\G^\circ$-modules are also the same, we conclude that the respective composition factors in the parabolic Verma modules coincide.
\end{rem}

\begin{cor}\label{cor:exceptional types}
The characters of the irreducible modules over the exceptional Lie
superalgebras $G(3)$, $F(3|1)$, and $D(2|1,\alpha)$ with
$\alpha\in\N$ in the BGG category, with respect to the distinguished
fundamental systems, are determined by those over symmetrizable
Kac-Moody Lie algebras.
\end{cor}

\begin{proof}
If we let $A^{\texttt{hd}}$ be the SGCM's
of the exceptional finite-dimensional Lie superalgebras $G(3)$,
$F(3|1)$, and $D(2|1,\alpha)$ in Example \ref{ex:except:simples},
then the respective Dynkin diagrams of Lie superalgebras $\DG$
become the following:\vskip 5mm

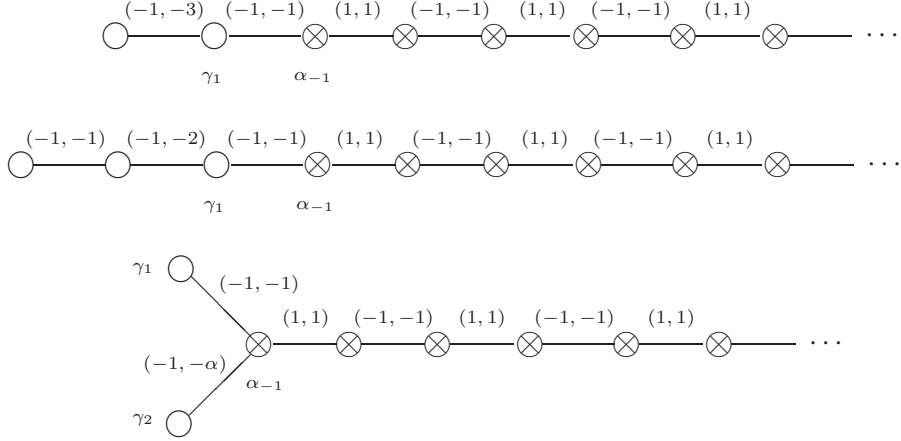
\begin{center}
\hskip 1.5cm \setlength{\unitlength}{0.22in}
\begin{picture}(20,2)
\put(0.50,1){\makebox(0,0)[c]{$\bigcirc$}}
\put(2.85,1){\makebox(0,0)[c]{$\bigcirc$}}
\put(5.25,1){\makebox(0,0)[c]{\Large $\otimes$}}
\put(7.4,1){\makebox(0,0)[c]{\Large $\otimes$}}
\put(11.7,1){\makebox(0,0)[c]{\Large $\otimes$}}
\put(9.5,1){\makebox(0,0)[c]{\Large $\otimes$}}
\put(14,1){\makebox(0,0)[c]{\Large $\otimes$}}
\put(16.2,1){\makebox(0,0)[c]{\Large $\otimes$}}
\put(18.8,1){\makebox(0,0)[c]{$\cdots$}}
\put(0.8,1){\line(1,0){1.7}}
\put(3.2,1){\line(1,0){1.7}}
\put(5.6,1){\line(1,0){1.5}}
\put(7.7,1){\line(1,0){1.5}}
\put(9.8,1){\line(1,0){1.5}}
\put(12,1){\line(1,0){1.7}}
\put(14.3,1){\line(1,0){1.5}}
\put(16.5,1){\line(1,0){1.5}}
\put(5.7,1.5){\tiny{$(1,1)$}}
\put(7.5,1.5){\tiny{$(-1,-1)$}}
\put(10.1,1.5){\tiny{$(1,1)$}}
\put(11.8,1.5){\tiny{$(-1,-1)$}}
\put(14.5,1.5){\tiny{$(1,1)$}}
\put(3.1,1.5){\tiny{$(-1,-1)$}}
\put(0.7,1.5){\tiny{$(-1,-3)$}}
\put(5.2,0){\makebox(0,0)[c]{\tiny $\alpha_{-1}$}}
\put(2.8,0){\makebox(0,0)[c]{\tiny $\gamma_{1}$}}
\end{picture}
\end{center}\vskip 3mm

\begin{center}
\hskip 1cm \setlength{\unitlength}{0.22in}
\begin{picture}(19,2)
\put(-1.80,1){\makebox(0,0)[c]{$\bigcirc$}}
\put(0.50,1){\makebox(0,0)[c]{$\bigcirc$}}
\put(2.85,1){\makebox(0,0)[c]{$\bigcirc$}}
\put(5.25,1){\makebox(0,0)[c]{\Large $\otimes$}}
\put(7.4,1){\makebox(0,0)[c]{\Large $\otimes$}}
\put(11.7,1){\makebox(0,0)[c]{\Large $\otimes$}}
\put(9.5,1){\makebox(0,0)[c]{\Large $\otimes$}}
\put(14,1){\makebox(0,0)[c]{\Large $\otimes$}}
\put(16.2,1){\makebox(0,0)[c]{\Large $\otimes$}}
\put(18.8,1){\makebox(0,0)[c]{$\cdots$}}
\put(-1.5,1){\line(1,0){1.7}}
\put(0.8,1){\line(1,0){1.7}}
\put(3.2,1){\line(1,0){1.7}}
\put(5.6,1){\line(1,0){1.5}}
\put(7.7,1){\line(1,0){1.5}}
\put(9.8,1){\line(1,0){1.5}}
\put(12,1){\line(1,0){1.7}}
\put(14.3,1){\line(1,0){1.5}}
\put(16.5,1){\line(1,0){1.5}}
\put(5.7,1.5){\tiny{$(1,1)$}}
\put(7.5,1.5){\tiny{$(-1,-1)$}}
\put(10.1,1.5){\tiny{$(1,1)$}}
\put(11.8,1.5){\tiny{$(-1,-1)$}}
\put(14.5,1.5){\tiny{$(1,1)$}}
\put(3.1,1.5){\tiny{$(-1,-1)$}}
\put(0.7,1.5){\tiny{$(-1,-2)$}}
\put(-1.7,1.5){\tiny{$(-1,-1)$}}
\put(5.2,0){\makebox(0,0)[c]{\tiny $\alpha_{-1}$}}
\put(2.8,0){\makebox(0,0)[c]{\tiny $\gamma_{1}$}}
\end{picture}
\end{center}\vskip 3mm

\begin{center}
\hskip 0cm \setlength{\unitlength}{0.22in}
\begin{picture}(20,3.2)
\put(3.4,2.8){\makebox(0,0)[c]{$\bigcirc$}}
\put(3.35,-0.9){\makebox(0,0)[c]{$\bigcirc$}}
\put(5.25,1){\makebox(0,0)[c]{\Large $\otimes$}}
\put(7.4,1){\makebox(0,0)[c]{\Large $\otimes$}}
\put(11.7,1){\makebox(0,0)[c]{\Large $\otimes$}}
\put(9.5,1){\makebox(0,0)[c]{\Large $\otimes$}}
\put(14,1){\makebox(0,0)[c]{\Large $\otimes$}}
\put(16.2,1){\makebox(0,0)[c]{\Large $\otimes$}}
\put(18.8,1){\makebox(0,0)[c]{$\cdots$}}
\put(3.65,2.6){\line(1,-1){1.4}}
\put(3.6,-0.7){\line(1,1){1.45}}
\put(5.6,1){\line(1,0){1.5}}
\put(7.7,1){\line(1,0){1.5}}
\put(9.8,1){\line(1,0){1.5}}
\put(12,1){\line(1,0){1.7}}
\put(14.3,1){\line(1,0){1.5}}
\put(16.5,1){\line(1,0){1.5}}
\put(5.8,1.5){\tiny{$(1,1)$}}
\put(7.5,1.5){\tiny{$(-1,-1)$}}
\put(10,1.5){\tiny{$(1,1)$}}
\put(11.8,1.5){\tiny{$(-1,-1)$}}
\put(14.5,1.5){\tiny{$(1,1)$}}
\put(4.3,2.3){\tiny{$(-1,-1)$}}
\put(2.5,0.4){\tiny{$(-1,-\alpha)$}}
\put(5.4,0){\makebox(0,0)[c]{\tiny $\alpha_{-1}$}}
\put(2.5,2.8){\makebox(0,0)[c]{\tiny $\gamma_{1}$}}
\put(2.5,-0.8){\makebox(0,0)[c]{\tiny $\gamma_{2}$}}
\end{picture}
\end{center}
\vskip 5mm

By Corollary~ \ref{crit:KM}, the corresponding Lie algebras $\G$ are
all symmetrizable Kac-Moody Lie algebras (of infinite rank). We note
that the Lie superalgebras $\SG$ contains subalgebras $\SG_{-1}$
which are isomorphic to these finite-dimensional
exceptional Lie superalgebras.
Finally, we apply Corollary~\ref{cor:KL} to complete the proof.
\end{proof}

\begin{rem}
The super dimensions for finite-dimensional irreducible modules over
these exceptional simple Lie superalgebras were studied earlier in
\cite{TM,KW1}. For $D(2|1,\alpha)$, finite-dimensional irreducible
character formulas were obtained  in \cite{Ger} (see also \cite{vdJ}). Recently, in
\cite{SZ} finite-dimensional irreducible character formulas for
all these Lie superalgebras were computed.
\end{rem}

\section{Super duality}\label{Super duality}

In this section we continue to work under the assumptions
\eqref{eq:pseudoCartan for hd} and \eqref{eq:locally nilpotent}
on  $\DG$, $\SG$, and $\G$ as in Section \ref{Irreducible characters}.
Let us first briefly recall the notion of Kostant (co)homology
groups of the nilradicals of $\DG$, $\SG$, and $\G$ with
coefficients in modules from $\wt{\mc O}$, $\ov{\mc O}$, and
$\mc{O}$, respectively (see, e.g., \cite[Section 6.4]{CW} and the
references therein).

For $\wt{M}\in \wt{\mc{O}}$, let $M=T(\wt{M})$ and
$\ov{M}=\ov{T}(\wt{M})$. We denote by $H_n(\wt{\mf u}^-,\wt{M})$ for
$n\in \Z_+$ the Kostant $\wt{\mf u}^-$-homology groups with
coefficients in $\wt{M}$, which are determined from the chain
complex $\wt{d}: \Lambda(\wt{\mf u}^-)\otimes \wt{M} \longrightarrow
\Lambda(\wt{\mf u}^-)\otimes \wt{M}$, where $\Lambda(\wt{\mf
u}^-)=\bigoplus_{n\geq 0}\La^n \wt{\mf u}^-$ denotes the super
exterior algebra generated by $\wt{\mf u}^-$. Note that the boundary operator $\wt{d}$ is
an $\wt{\mf l}$-module homomorphism, and $H_n(\wt{\mf u}^-,\wt{M})$
is a semisimple $\wt{\mf l}$-module since $\Lambda(\wt{\mf u}^-)$
and $\Lambda(\wt{\mf u}^-)\otimes \wt{M}$ are semisimple over
$\wt{\mf l}$. The homology groups $H_n(\ov{\mf u}^-,\ov{M})$ and
$H_n({\mf u}^-,{M})$ are defined by the chain complexes $\ov{d}:
\Lambda(\ov{\mf u}^-)\otimes \ov{M} \longrightarrow \Lambda(\ov{\mf
u}^-)\otimes \ov{M}$ and ${d}: \Lambda({\mf u}^-)\otimes {M}
\longrightarrow \Lambda({\mf u}^-)\otimes {M}$, which are semisimple
over $\ov{\mf l}$ and $\mf{l}$, respectively.

The cohomology groups $H^n(\wt{\mf u}^+,\wt{M})$ for $n\in \Z_+$ are
defined by the cochain complex $\wt{\partial} : C(\wt{\mf
u}^+,\wt{M})\longrightarrow C(\wt{\mf u}^+,\wt{M})$, where
$C(\wt{\mf u}^+,\wt{M})= \bigoplus_{n\geq 0} {\rm Hom}(\La^n \wt{\mf
u}^+ ,\wt{M})$ and ${\rm Hom}(\La^n \wt{\mf u}^+ ,\wt{M})=(\La^n
\wt{\mf u}^+)^*\otimes \wt{M}$ with $(\La^n \wt{\mf u}^+)^*$
denoting the restricted dual of $\La^n \wt{\mf u}^+$. The cohomology
groups $H^n(\ov{\mf u}^+,\ov{M})$ and $H^n({\mf u}^+,{M})$ are
defined similarly.

Let $\tau$ be the automorphism of $\DG$ given by
$\tau(e_i)=-(-1)^{p(e_i)}f_i$, $\tau(f_i)=-e_i$ and $\tau(h)=-h$ for
$i\in \wt{I}$ and $h\in \wt{\h}$. We denote by $\wt{M}^\vee$ the
restricted dual $\wt{M}^*=\bigoplus_{\mu}\wt{M}^*_\mu$ with
$\tau$-twisted $\DG$-action, that is, $(x\cdot
f)(v)=(-1)^{p(x)p(f)+1}f(\tau(x)v)$ for $v\in \wt{M}$, and
homogeneous elements $x\in \DG$, $f\in \wt{M}^*$. Then we have the
following duality:
\begin{equation}\label{Poincare duality}
H_n(\wt{\mf u}^-,\wt{M})\cong H^n(\wt{\mf u}^+,\wt{M}^\vee),
\end{equation}
for $n\in \Z_+$ as $\wt{\mf l}$-modules (cf. \cite[Theorem 6.24]{CW}
and \cite{L}). Similar dualities also hold for $H_n(\ov{\mf u}^-,\ov{M})$
and $H_n({\mf u}^-,{M})$.

\begin{thm}\label{thm:Kostant} For $\wt{M}\in \wt{\mc O}$ and $n\in\Z_+$, we have
\begin{itemize}
\item[(1)] $T(H_n(\wt{\mf u}^-,\wt{M}))\cong H_n({\mf u}^-,{M})$ as $\mf{l}$-modules,

\item[(2)]
$\ov{T}(H_n(\wt{\mf u}^-,\wt{M}))\cong H_n(\ov{\mf u}^-,\ov{M})$ as
$\ov{\mf{l}}$-modules.
\end{itemize}
\end{thm}

\begin{proof}
By Lemma \ref{borel2borel}, we have $T(\Lambda(\wt{\mf u}^-)\otimes
\wt{M})= \Lambda({\mf u}^-)\otimes {M}$ and $\ov{T}(\Lambda(\wt{\mf
u}^-)\otimes \wt{M})= \Lambda(\ov{\mf u}^-)\otimes \ov{M}$  with
$T[\wt{d}]=d$ and $\ov{T}[\wt{d}]=\ov{d}$, respectively. Since the
chain complexes are semisimple over the Levi subalgebras, we have
isomorphisms between the homology groups by Lemma \ref{lem:chL} and
Proposition \ref{prop:T is exact}.
\end{proof}

\begin{cor}\label{cor:Kostant}
For $\la\in P^+$ and $n\in \Z_+$, we have
\begin{itemize}
\item[(1)]
$T(H_n(\wt{\mf u}^-,\wt{L}(\la^\theta)))\cong H_n({\mf
u}^-,{L}(\lambda))$ as $\mf{l}$-modules,

\item[(2)]
$\ov{T}(H_n(\wt{\mf u}^-,\wt{L}(\la^\theta)))\cong H_n(\ov{\mf
u}^-,\ov{L}(\la^\natural))$ as $\ov{\mf{l}}$-modules.
\end{itemize}
\end{cor}

\begin{cor}\label{cor:Kostant-2}
For $\la\in P^+$ and $n\in \Z_+$, we have
\begin{itemize}
\item[(1)]
$T(H^n(\wt{\mf u}^+,\wt{L}(\la^\theta)))\cong H^n({\mf
u}^+,{L}(\lambda))$ as $\mf{l}$-modules,

\item[(2)]
$\ov{T}(H^n(\wt{\mf u}^+,\wt{L}(\la^\theta)))\cong H^n(\ov{\mf
u}^+,\ov{L}(\la^\natural))$ as $\ov{\mf{l}}$-modules.
\end{itemize}
\end{cor}
\begin{proof}
Note that $\wt{L}(\la^\theta)$, $\ov{L}(\la^\natural)$, and $L(\la)$
are self-dual with respect to the functor $\vee$. Then the
isomorphisms follow from \eqref{Poincare duality} and Corollary
\ref{cor:Kostant}.
\end{proof}

Based on \cite{KL,V}, we define the (parabolic)
Kazhdan-Lusztig-Vogan polynomials in  $\mc O$, $\ov{\mc O}$, and
$\wt{\mc{O}}$ for $\la, \mu\in P^+$ by
\begin{equation*}
\begin{split}
\ell_{\mu \la}(q)
&:=\sum_{n=0}^\infty (-q)^{-n}\dim{\rm Hom}_{\mf l}\left( L(\mf{l},\mu),
H_n(\mf{u}^-, L(\la)) \right),
 \\
\ov{\ell}_{\mu^\natural \la^\natural}(q)
&:=\sum_{n=0}^\infty (-q)^{-n} \dim{\rm Hom}_{\ov{\mf l}}
\left( \ov{L}(\ov{\mf{l}},\mu^\natural), H_n(\ov{\mf{u}}^-, \ov{L}(\la^\natural)) \right),
 \\
\wt{\ell}_{\mu^\theta \la^\theta}(q)
&:=\sum_{n=0}^\infty (-q)^{-n}\dim{\rm Hom}_{\wt{\mf l}}
\left( \wt{L}(\wt{\mf{l}},\mu^\theta), H_n(\wt{\mf{u}}^-, \wt{L}(\la^\theta)) \right).
\end{split}
\end{equation*}
By standard arguments, we see that ${\rm ch}\,
L(\lambda)=\sum_{\mu\in P^+}\ell_{\mu\la}(1){\rm ch}\, \Delta(\mu)$.
By Theorem \ref{thm:Kostant}, we have the following.
\begin{thm}
For $\la, \mu\in P^+$, we have $\ell_{\mu
\la}(q)=\ov{\ell}_{\mu^\natural
\la^\natural}(q)=\wt{\ell}_{\mu^\theta \la^\theta}(q)$.
\end{thm}

Now, we are ready to state one main result in this paper, which
naturally extends the results in \cite{CL} and \cite{CLW}.

\begin{thm}\label{thm:super duality}\mbox{}
The following statements hold.
\begin{itemize}
\item[(1)]
$T : \wt{\mc O}\longrightarrow {\mc O}$ is an equivalence of categories.

\item[(2)]
$\ov{T} : \wt{\mc O}\longrightarrow \ov{\mc O}$ is an equivalence of categories.

\item[(3)] The categories ${\mc O}$ and $\ov{\mc O}$ are equivalent.
\end{itemize}
\end{thm}

\begin{proof}
Let us define an equivalence relation $\sim$ on $\wt{\h}^*$ by
letting $\mu\sim \nu$ if and only if $\mu-\nu\in\sum_{\alpha\in
\wt{\Pi}}\Z\alpha$. For $\mu\in \wt{\h}^*$, we fix a representative
$[\mu]^o$ in the equivalence class $[\mu]$ and declare  the parity
of $[\mu]^o$ is $\ov{0}$. Consider a $\Z_2$-grading on $\wt{\h}^*$
as follows:
\begin{equation}\label{Grading on Cartan}
\wt{\h}^*_{\varepsilon}=\{\,\mu\in \wt{\h}^* \,\vert\
p\left(\mu-[\mu]^o\right) \equiv \varepsilon \!\!\!\pmod 2\,\} \ \ \
\ (\varepsilon\in\Z_2).
\end{equation}
Let $\wt{\mc O}^{\ov{0}}$ be the full subcategory of $\wt{\mc O}$
consisting of $\DG$-modules equipped with the $\Z_2$-grading in
\eqref{Grading on Cartan}. The subcategories $\ov{\mc O}^{\ov{0}}$
and ${\mc O}^{\ov{0}}$ are defined similarly. Since $T : \wt{\mc
O}^{\ov{0}}\longrightarrow {\mc O}^{\ov 0}$, $\ov{T} : \wt{\mc
O}^{\ov{0}}\longrightarrow \ov{\mc O}^{\ov{0}}$, and $\wt{\mc
O}^{\ov{0}}$ is equivalent to $\wt{\mc O}$, we may assume without
loss of generality that the modules in $\wt{\mc O}$, $\ov{\mc O}$,
and $\mc O$ are equipped with the $\Z_2$-grading in \eqref{Grading
on Cartan}. In particular, all categories involved are abelian.

Now, with Corollaries \ref{cor:Kostant} and \ref{cor:Kostant-2} at
our disposal, we may apply the same arguments as in \cite[Section
6.5]{CW} (cf. \cite{CL,CLW}).

First, we can prove that for $\wt{M}, \wt{N}\in \wt{\mc O}$, $T$ induces isomorphisms
\begin{itemize}
\item[(i)]
${\rm Hom}_{\wt{\mc O}}(\wt{M},\wt{N})\cong {\rm Hom}_{{\mc O}}({M},{N})$,

\item[(ii)]
${\rm Ext}^1_{\wt{\mc O}}(\wt{M},\wt{N})\cong {\rm Ext}^1_{{\mc
O}}({M},{N})$ with $\wt{M}$ a highest weight module,
\end{itemize}
where $M=T(\wt{M})$ and $N=T(\wt{N})$. Next, let $M\in {\mc O}$ be
given. By standard arguments, we have a filtration $0=M_0\subset
M_1\subset M_2\subset \cdots$ such that $M=\bigcup_{i\geq 1}M_i$ and
$M_i/M_{i-1}$ is a highest weight $\G$-module. We can also prove
that for a highest weight $\G$-module $V$ with highest weight $\la$,
there exists a highest weight $\DG$-module $\wt{V}$ with highest
weight $\la^\theta$ such that $T(\wt{V})\cong V$. This, together
with (ii), implies that there exists a filtration of $\DG$-modules
in $\wt{\mc O}$,  $0=\wt{M}_0\subset \wt{M}_1\subset \wt{M}_2\subset
\cdots$ such that $T(\wt{M}_i)\cong M_i$, and hence $T(\wt{M})\cong
M$, where $\wt{M}=\bigcup_{i\geq 1}\wt{M}_i$. Finally, by (i), we
conclude that $T : \wt{\mc O}\longrightarrow {\mc O}$ is an
equivalence of categories, which proves (1). The proof of (2) is
parallel. Now, the equivalence between $\ov{\mc O}$ and ${\mc O}$
follows from (1) and (2).
\end{proof}

The equivalence in Theorem \ref{thm:super duality} (3) is called
{\it super duality}. It was conjectured in \cite{CWZ,CW1} and proved
in \cite{CL} for \framebox{${\mf B}$} of type $A_n$, and then proved
in \cite{CLW} for \framebox{${\mf B}$} of type $B_n$, $C_n$, $D_n$,
and $B(0,n)$.
Theorem \ref{thm:super duality} holds also when \framebox{${\mf B}$} is of infinite rank.

\begin{rem}\label{more than one tail}
We may also modify the setup at the beginning of Section \ref{Three
Lie superalgebras} to have equivalences for a more general class of
$\DG$, where more than one odd isotropic simple root, say
$\alpha^{(1)}_{-1},\ldots,\alpha^{(u)}_{-1}$, together with its
associated tail diagram of $\gl(\infty|\infty)$, is attached to
\framebox{${\mf B}$} as in \eqref{DG diagram} (cf.~\cite{CLW2}). Here, a pair of
vertices $\alpha^{(i)}_{-1}$ and $\alpha^{(j)}_{-1}$ may have an
edge with label $(d_{ij},d_{ji})$ for $1\leq i<j\leq u$ (not
necessarily being negative), and a vertex $\gamma$ in
\framebox{$\mf{B}$} may be connected to more than one vertex from
$\alpha^{(1)}_{-1},\ldots,\alpha^{(u)}_{-1}$.
% In this case, the levi subalgebra is $\wt{\mf l}=\h + \DG^{\texttt t}$,
%where $\DG^{\texttt t}\cong \gl(\infty|\infty)^{\oplus u}$

For example, when $u=2$, we have equivalences between $\wt{\mc O}$,
$\ov{\mc O}$, and $\mc O$ for
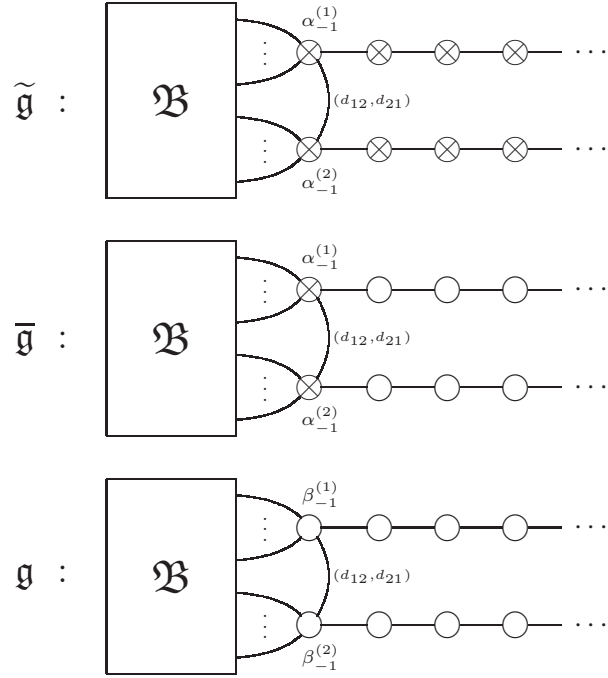
\begin{center}
\hskip -2cm \setlength{\unitlength}{0.17in}
\begin{picture}(20,7)
\put(10.25,1.5){\makebox(0,0)[c]{\Large $\otimes$}}
\put(12.4,1.5){\makebox(0,0)[c]{\Large $\otimes$}}
\put(19,1.5){\makebox(0,0)[c]{$\cdots$}}
%\put(21.7,1.5){\makebox(0,0)[c]{$\cdots$}}
\put(10.6,.5){\makebox(0,0)[c]{\tiny $\alpha^{(2)}_{-1}$}}
\qbezier(9.95,1.8)(9.4,2.4)(8,2.5)
\qbezier(9.95,1.2)(9.4,0.6)(8,0.5)
\qbezier(10.5,1.8)(11.2,3)(10.5,4.25)
\put(12.2,3){\makebox(0,0)[c]{\tiny$_{(d_{12},d_{21})}$}}
\put(10.6,1.5){\line(1,0){1.45}}
\put(12.8,1.5){\line(1,0){1.3}}
\put(14.9,1.5){\line(1,0){1.3}}
\put(17,1.5){\line(1,0){1}}
%\put(19.9,1.5){\line(1,0){1}}
%\put(9,0.5){\line(1,2){1.1}}
%\put(8.4,2.7){\tiny{$_{(b_1,c_1)}$}}
\put(8.8,2.1){\tiny{$_{\vdots}$}}
%\put(8.4,0.3){\tiny{$_{(b_p,c_p)}$}}
%
\put(14.5,1.5){\makebox(0,0)[c]{\Large $\otimes$}}
\put(16.6,1.5){\makebox(0,0)[c]{\Large $\otimes$}}
%
%\put(10.8,1.9){\tiny{$_{(1,1)}$}}
%\put(12.65,1.9){\tiny{$_{(-1,-1)}$}}
%\put(15.2,1.9){\tiny{$_{(1,1)}$}}
%
\put(10.25,4.5){\makebox(0,0)[c]{\Large $\otimes$}}
\put(12.4,4.5){\makebox(0,0)[c]{\Large $\otimes$}}
\put(19,4.5){\makebox(0,0)[c]{$\cdots$}}
%\put(21.7,4.5){\makebox(0,0)[c]{$\cdots$}}
\put(10.6,5.5){\makebox(0,0)[c]{\tiny $\alpha^{(1)}_{-1}$}}
\qbezier(9.95,4.8)(9.4,5.4)(8,5.5)
\qbezier(9.95,4.2)(9.4,3.6)(8,3.5)
\put(10.6,4.5){\line(1,0){1.45}}
\put(12.8,4.5){\line(1,0){1.3}}
\put(14.9,4.5){\line(1,0){1.3}}
\put(17,4.5){\line(1,0){1}}
%\put(19.9,4.5){\line(1,0){1}}
%\put(9,0.5){\line(1,2){1.1}}
%\put(8.4,2.7){\tiny{$_{(b_1,c_1)}$}}
\put(8.8,5.1){\tiny{$_{\vdots}$}}
%\put(8.4,0.3){\tiny{$_{(b_p,c_p)}$}}
%
\put(14.5,4.5){\makebox(0,0)[c]{\Large $\otimes$}}
\put(16.6,4.5){\makebox(0,0)[c]{\Large $\otimes$}}
\put(8,0){\line(0,1){6}}
\put(4,0){\line(0,1){6}}
\put(8,6){\line(-1,0){4}}
\put(8,0){\line(-1,0){4}}
\put(6,3){\makebox(0,0)[c]{{\LARGE$\mf{B}$}}}
\put(2,3){\makebox(0,0)[c]{{\Large$\DG\ :$}}}
\end{picture}
\end{center}

\begin{center}
\hskip -2cm \setlength{\unitlength}{0.17in}
\begin{picture}(20,7)
\put(10.25,1.5){\makebox(0,0)[c]{\Large $\otimes$}}
\put(12.4,1.5){\makebox(0,0)[c]{$\bigcirc$}}
\put(19,1.5){\makebox(0,0)[c]{$\cdots$}}
%\put(21.7,1.5){\makebox(0,0)[c]{$\cdots$}}
\put(10.6,.5){\makebox(0,0)[c]{\tiny $\alpha^{(2)}_{-1}$}}
\qbezier(9.95,1.8)(9.4,2.4)(8,2.5)
\qbezier(9.95,1.2)(9.4,0.6)(8,0.5)
\qbezier(10.5,1.8)(11.2,3)(10.5,4.25)
\put(12.2,3){\makebox(0,0)[c]{\tiny$_{(d_{12},d_{21})}$}}
\put(10.6,1.5){\line(1,0){1.45}}
\put(12.8,1.5){\line(1,0){1.3}}
\put(14.9,1.5){\line(1,0){1.3}}
\put(17,1.5){\line(1,0){1}}
%\put(19.9,1.5){\line(1,0){1}}
%\put(9,0.5){\line(1,2){1.1}}
%\put(8.4,2.7){\tiny{$_{(b_1,c_1)}$}}
\put(8.8,2.1){\tiny{$_{\vdots}$}}
%\put(8.4,0.3){\tiny{$_{(b_p,c_p)}$}}
%
\put(14.5,1.5){\makebox(0,0)[c]{$\bigcirc$}}
\put(16.6,1.5){\makebox(0,0)[c]{$\bigcirc$}}
%
%\put(10.8,1.9){\tiny{$_{(1,1)}$}}
%\put(12.65,1.9){\tiny{$_{(-1,-1)}$}}
%\put(15.2,1.9){\tiny{$_{(1,1)}$}}
%
\put(10.25,4.5){\makebox(0,0)[c]{\Large $\otimes$}}
\put(12.4,4.5){\makebox(0,0)[c]{$\bigcirc$}}
\put(19,4.5){\makebox(0,0)[c]{$\cdots$}}
%\put(21.7,4.5){\makebox(0,0)[c]{$\cdots$}}
\put(10.6,5.5){\makebox(0,0)[c]{\tiny $\alpha^{(1)}_{-1}$}}
\qbezier(9.95,4.8)(9.4,5.4)(8,5.5)
\qbezier(9.95,4.2)(9.4,3.6)(8,3.5)
\put(10.6,4.5){\line(1,0){1.45}}
\put(12.8,4.5){\line(1,0){1.3}}
\put(14.9,4.5){\line(1,0){1.3}}
\put(17,4.5){\line(1,0){1}}
%\put(19.9,4.5){\line(1,0){1}}
%\put(9,0.5){\line(1,2){1.1}}
%\put(8.4,2.7){\tiny{$_{(b_1,c_1)}$}}
\put(8.8,5.1){\tiny{$_{\vdots}$}}
%\put(8.4,0.3){\tiny{$_{(b_p,c_p)}$}}
%
\put(14.5,4.5){\makebox(0,0)[c]{$\bigcirc$}}
\put(16.6,4.5){\makebox(0,0)[c]{$\bigcirc$}}
\put(8,0){\line(0,1){6}}
\put(4,0){\line(0,1){6}}
\put(8,6){\line(-1,0){4}}
\put(8,0){\line(-1,0){4}}
\put(6,3){\makebox(0,0)[c]{{\LARGE$\mf{B}$}}}
\put(2,3){\makebox(0,0)[c]{{\Large$\SG\ :$}}}
\end{picture}
\end{center}

\begin{center}
\hskip -2cm \setlength{\unitlength}{0.17in}
\begin{picture}(20,7)
\put(10.25,1.5){\makebox(0,0)[c]{$\bigcirc$}}
\put(12.4,1.5){\makebox(0,0)[c]{$\bigcirc$}}
\put(19,1.5){\makebox(0,0)[c]{$\cdots$}}
%\put(21.7,1.5){\makebox(0,0)[c]{$\cdots$}}
\put(10.6,.5){\makebox(0,0)[c]{\tiny $\beta^{(2)}_{-1}$}}
\qbezier(10,1.8)(9.4,2.4)(8,2.5)
\qbezier(10,1.2)(9.4,0.6)(8,0.5)
\qbezier(10.5,1.8)(11.2,3)(10.5,4.2)
\put(12.2,3){\makebox(0,0)[c]{\tiny$_{(d_{12},d_{21})}$}}
\put(10.6,1.5){\line(1,0){1.45}}
\put(12.8,1.5){\line(1,0){1.3}}
\put(14.9,1.5){\line(1,0){1.3}}
\put(17,1.5){\line(1,0){1}}
%\put(19.9,1.5){\line(1,0){1}}
%\put(9,0.5){\line(1,2){1.1}}
%\put(8.4,2.7){\tiny{$_{(b_1,c_1)}$}}
\put(8.8,2.1){\tiny{$_{\vdots}$}}
%\put(8.4,0.3){\tiny{$_{(b_p,c_p)}$}}
%
\put(14.5,1.5){\makebox(0,0)[c]{$\bigcirc$}}
\put(16.6,1.5){\makebox(0,0)[c]{$\bigcirc$}}
%
%\put(10.8,1.9){\tiny{$_{(1,1)}$}}
%\put(12.65,1.9){\tiny{$_{(-1,-1)}$}}
%\put(15.2,1.9){\tiny{$_{(1,1)}$}}
%
\put(10.25,4.5){\makebox(0,0)[c]{$\bigcirc$}}
\put(12.4,4.5){\makebox(0,0)[c]{$\bigcirc$}}
\put(19,4.5){\makebox(0,0)[c]{$\cdots$}}
%\put(21.7,4.5){\makebox(0,0)[c]{$\cdots$}}
\put(10.6,5.5){\makebox(0,0)[c]{\tiny $\beta^{(1)}_{-1}$}}
\qbezier(10,4.8)(9.4,5.4)(8,5.5)
\qbezier(10,4.2)(9.4,3.6)(8,3.5)
\put(10.6,4.5){\line(1,0){1.45}}
\put(12.8,4.5){\line(1,0){1.3}}
\put(14.9,4.5){\line(1,0){1.3}}
\put(17,4.5){\line(1,0){1}}
%\put(19.9,4.5){\line(1,0){1}}
%\put(9,0.5){\line(1,2){1.1}}
%\put(8.4,2.7){\tiny{$_{(b_1,c_1)}$}}
\put(8.8,5.1){\tiny{$_{\vdots}$}}
%\put(8.4,0.3){\tiny{$_{(b_p,c_p)}$}}
%
\put(14.5,4.5){\makebox(0,0)[c]{$\bigcirc$}}
\put(16.6,4.5){\makebox(0,0)[c]{$\bigcirc$}}
\put(8,0){\line(0,1){6}}
\put(4,0){\line(0,1){6}}
\put(8,6){\line(-1,0){4}}
\put(8,0){\line(-1,0){4}}
\put(6,3){\makebox(0,0)[c]{{\LARGE$\mf{B}$}}}
\put(2,3){\makebox(0,0)[c]{{\Large$\G\ :$}}}
\end{picture}
\end{center}\vskip 5mm\end{rem}

\begin{example}\label{ex:affine super}
Suppose that \framebox{$\mf{B}$} is a disjoint union of the
following two Dynkin diagrams of type $A$:
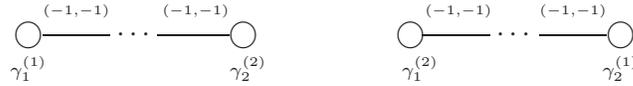
\begin{center}
\hskip -4cm \setlength{\unitlength}{0.25in}
\begin{picture}(24,2)
\put(8.2,1){\makebox(0,0)[c]{$\bigcirc$}}
\put(10.45,1){\makebox(0,0)[c]{$\cdots$}}
\put(12.7,1){\makebox(0,0)[c]{$\bigcirc$}}
\put(16.2,1){\makebox(0,0)[c]{$\bigcirc$}}
\put(18.4,1){\makebox(0,0)[c]{$\cdots$}}
\put(20.6,1){\makebox(0,0)[c]{$\bigcirc$}}
\put(8.5,1){\line(1,0){1.4}}
\put(10.9,1){\line(1,0){1.5}}
\put(8.5,1.5){\tiny{$_{(-1,-1)}$}}
\put(11,1.5){\tiny{$_{(-1,-1)}$}}
\put(16.5,1.5){\tiny{$_{(-1,-1)}$}}
\put(18.9,1.5){\tiny{$_{(-1,-1)}$}}
\put(16.45,1){\line(1,0){1.4}}
\put(18.9,1){\line(1,0){1.4}}
\put(8.2,0.3){\makebox(0,0)[c]{\tiny $\gamma^{(1)}_{1}$}}
\put(20.6,0.3){\makebox(0,0)[c]{\tiny $\gamma^{(1)}_2$}}
\put(12.8,0.3){\makebox(0,0)[c]{\tiny $\gamma^{(2)}_{2}$}}
\put(16.4,0.3){\makebox(0,0)[c]{\tiny $\gamma^{(2)}_1$}}
\end{picture}
\end{center}
Here $\gamma^{(1)}_1$, $\gamma^{(1)}_2$, $\gamma^{(2)}_1$, and
$\gamma^{(2)}_2$  are the four end vertices. We connect
\framebox{$\mf{B}$} to two vertices $\alpha^{(1)}_{-1}$ and
$\alpha^{(2)}_{-1}$ as follows. We connect to $\alpha^{(i)}_{-1}$
the two vertices $\gamma^{(i)}_1$ and $\gamma^{(i)}_2$ with labels
$(b^{(i)}_1,c^{(i)}_1)=(-1,1)$ and $(b^{(i)}_2,c^{(i)}_2)=(-1,-1)$
so that the resulting head diagram is of the form:
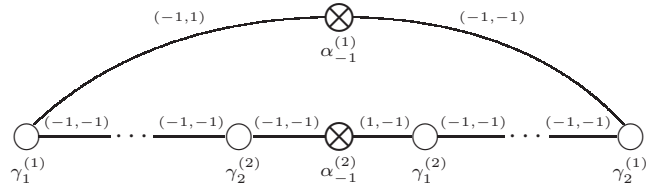
\begin{center}
\hskip -4cm \setlength{\unitlength}{0.25in}
\begin{picture}(24,4)
\put(7.95,1){\makebox(0,0)[c]{$\bigcirc$}}
\put(10.2,1){\makebox(0,0)[c]{$\cdots$}}
\put(12.4,1){\makebox(0,0)[c]{$\bigcirc$}}
\put(16.3,1){\makebox(0,0)[c]{$\bigcirc$}}
\put(18.45,1){\makebox(0,0)[c]{$\cdots$}}
\qbezier(8.1,1.25)(10.4,3.4)(14.2,3.5)
\qbezier(14.8,3.5)(18.45,3.4)(20.45,1.25)
\put(8.2,1){\line(1,0){1.5}}
\put(10.6,1){\line(1,0){1.5}}
\put(8.3,1.3){\tiny{$_{(-1,-1)}$}}
\put(10.7,1.3){\tiny{$_{(-1,-1)}$}}
\put(12.7,1){\line(1,0){1.5}}
\put(12.7,1.3){\tiny{$_{(-1,-1)}$}}
\put(14.8,1){\line(1,0){1.2}}
\put(14.9,1.3){\tiny{$_{(1,-1)}$}}
\put(16.6,1.3){\tiny{$_{(-1,-1)}$}}
\put(18.8,1.3){\tiny{$_{(-1,-1)}$}}
\put(17,3.5){\tiny{$_{(-1,-1)}$}}
\put(10.6,3.5){\tiny{$_{(-1,1)}$}}
\put(16.55,1){\line(1,0){1.4}}
\put(18.9,1){\line(1,0){1.4}}
\put(14.5,1){\makebox(0,0)[c]{$\bigotimes$}}
\put(14.5,3.5){\makebox(0,0)[c]{$\bigotimes$}}
\put(20.6,1){\makebox(0,0)[c]{$\bigcirc$}}
\put(14.5,2.8){\makebox(0,0)[c]{\tiny $\alpha^{(1)}_{-1}$}}
\put(14.5,0.3){\makebox(0,0)[c]{\tiny $\alpha^{(2)}_{-1}$}}
\put(8,0.3){\makebox(0,0)[c]{\tiny $\gamma_{1}^{(1)}$}}
\put(20.6,0.3){\makebox(0,0)[c]{\tiny $\gamma_2^{(1)}$}}
\put(12.5,0.3){\makebox(0,0)[c]{\tiny $\gamma_{2}^{(2)}$}}
\put(16.4,0.3){\makebox(0,0)[c]{\tiny $\gamma_1^{(2)}$}}
\end{picture}
\end{center}
This gives the Dynkin diagram of the affine Lie superalgebra of type $A$.
\end{example}

%%%%%
%%%%%
%%%%%
%%%%%
\section{Integrable modules for Kac-Moody Lie superalgebras}
 \label{sec:integrable}

In this section, we assume that  $\G(C)$ is a
Kac-Moody Lie superalgebra associated with  $C=(c_{ij})_{i,j\in I}$ satisfying the mild condition
\begin{equation}
\label{eq:neq0}\tag{{\bf C}}
\text{$C$ is a symmetrizable  SGCM with $c_{ij}\leq 0$\ \  for all $i\neq j$.}
\end{equation}
Recall $I =I_{\ov{0}} \sqcup I_{\ov{1}} $. For convenience, we set %
\begin{align*}
\Icirc=\{\,i\in I_{\ov{0}} \mid c_{ii}=2\,\}, \quad
&
I_{\otimes}=\{\,i\in I\,|\,c_{ii}=0 \,\},
 \\
  \Ibullet =\{\,i\in I_{\ov{1}} \mid c_{ii}=2\,\},
 \quad
& I_{\text{\RIGHTcircle}}=\Icirc \sqcup\Ibullet =\{\,i\in I \mid c_{ii}=2\,\}.
\end{align*}
We define ${\mc G}:=\G(C)\oplus \bigoplus_{i\in I_{\otimes}}\mathbb{C}d^{(i)}$,  an extension of $\G(C)$ by outer derivations, where $d^{(i)}$ is defined in the same way as $d$ corresponding to $\alpha_{-1}$ for $\DG$ in Section \ref{Three Lie superalgebras}.
Let $\Pi_{\mc G}$ be the
fundamental system associated with $C$, and let ${\mc O}_{\mc G}$ be
the BGG category with respect to $\Pi_{\mc G}$. Then we have the
following generalization of Corollary \ref{cor:exceptional types}.

\begin{thm}\label{thm:super duality-2}
The characters of irreducible $\mc G$-modules in ${\mc O}_{\mc G}$
are determined by those of irreducible modules over a symmetrizable
anisotropic Kac-Moody Lie superalgebra.
\end{thm}

\begin{proof}
Let $B$ be the submatrix of $C$ associated with $I_{\text{\RIGHTcircle}}$. For
$i\in I_{\otimes}$, let $\alpha^{(i)}_{-1}$ denote the corresponding
odd isotropic simple root of ${\mc G}$.  We consider the Kac-Moody
Lie superalgebra $\DG$  of infinite rank, whose Dynkin diagram is
obtained from \framebox{$\mf B$} by attaching diagrams of
$\gl(\infty|\infty)$ starting from $\alpha^{(i)}_{-1}$ ($i\in
I_{\otimes}$) in such a way that the head diagram of $\DG$ (cf.
\eqref{Head diagram})  is the Dynkin diagram \framebox{$\mf C$} of
$\mc G$ (see Remark \ref{more than one tail}). We define $\SG$ and
$\G$ in a similar way. Since $c_{ij}\leq 0$ for all $i\neq j$, $\G$
is a symmetrizable anisotropic Kac-Moody Lie superalgebra by
Proposition \ref{crit:KM}.

Consider the associated equivalent categories $\wt{\mc O}$, $\ov{\mc
O}$ and $\mc O$ with $J\subseteq I_{\text{\RIGHTcircle}}$ being empty. By
construction, we see that the truncated subalgebra $\SG_{-1}$ of
$\SG$ is isomorphic to ${\mc G}$. In the case of the example in
Remark \ref{more than one tail}, we have \vskip 5mm
\begin{center}
\hskip -2cm \setlength{\unitlength}{0.17in}
\begin{picture}(13,6)
\put(10.25,1.5){\makebox(0,0)[c]{\Large $\otimes$}}
\put(10.6,.5){\makebox(0,0)[c]{\tiny $\alpha^{(2)}_{-1}$}}
\qbezier(9.95,1.8)(9.4,2.4)(8,2.5)
\qbezier(9.95,1.2)(9.4,0.6)(8,0.5)
\qbezier(10.5,1.8)(11.2,3)(10.5,4.25)
%\put(12.2,3){\makebox(0,0)[c]{\tiny$_{(d_{12},d_{21})}$}}
%
\put(8.8,2.1){\tiny{$_{\vdots}$}}
\put(10.25,4.5){\makebox(0,0)[c]{\Large $\otimes$}}
\put(10.6,5.5){\makebox(0,0)[c]{\tiny $\alpha^{(1)}_{-1}$}}
\qbezier(9.95,4.8)(9.4,5.4)(8,5.5)
\qbezier(9.95,4.2)(9.4,3.6)(8,3.5)%
\put(8.8,5.1){\tiny{$_{\vdots}$}}
\put(8,0){\line(0,1){6}}
\put(4,0){\line(0,1){6}}
\put(8,6){\line(-1,0){4}}
\put(8,0){\line(-1,0){4}}
\put(6,3){\makebox(0,0)[c]{{\LARGE$\mf{B}$}}}
\put(2,3){\makebox(0,0)[c]{{\Large$\SG_{-1}\ :$}}}
\put(13,3){\makebox(0,0)[c]{{$= $}}}
\put(15,3){\makebox(0,0)[c]{\framebox{\Large$\mf C $}}}
\end{picture}
\end{center}
Moreover, the truncated category $\ov{\mc O}_{-1}$ is the BGG
category ${\mc O}_{\mc G}$ of ${\mc G}$-modules with respect to
$\Pi_{\mc G}$. Therefore, by Theorem  \ref{thm:irred:char} and
\eqref{truncation}, the character of an irreducible ${\mc G}$-module
in ${\mc O}_{\mc G}$ can be obtained from that of an irreducible
$\G$-module in $\mc O$.
\end{proof}

\begin{rem}\label{rem:ch of G}
When $\Ibullet=\emptyset$ in $C$,
$\G$ becomes a Kac-Moody Lie algebra,
and the
characters for a large class of irreducible $\G$-modules are
given by the Kazhdan-Lusztig polynomials of $\G$  (see \cite{KaTa2});
and see Remark~ \ref{rem:B} for the more general case. Applying the super duality functor we
obtain irreducible character formulas for the corresponding $\mc
G$-modules in ${\mc O}_{\mc G}$.
\end{rem}

\begin{rem}
Affine Lie superalgebras are related to Lie algebras $\G$, which are
not Kac-Moody Lie algebras, since the matrices associated to $\G$ have
a positive off-diagonal entry (see Examples \ref{example:affineosp} and \ref{ex:affine super}),
and hence they cannot be treated as in Theorem \ref{thm:super
duality-2} and Remark~ \ref{rem:ch of G}.
\end{rem}

We keep the notations for the Lie (super)algebras $\DG$, $\SG$, and
$\G$  associated with $\mc G$ as in the proof of Theorem~
\ref{thm:super duality-2}.

Let
\begin{equation*}
\wt{\Pi}=\{\,\alpha_i\,|\,i\in I_{\text{\RIGHTcircle}}\,\}\cup
\{\,\alpha^{(i)}_j\,|\,i\in I_{\otimes},\, j\in T_\infty\,\}
\end{equation*}
be the fundamental system of $\DG$, where $\{\,\alpha^{(i)}_j\,|\,\,
j\in T_\infty\,\}$ is the set of simple roots of
$\gl(\infty|\infty)$ for each $i\in I_{\otimes}$. Then the set of
integral weights $\wt{P}$ of $\DG$ is given by
\begin{equation*}
\wt{P}:=\sum_{i\in I_{\otimes}}\Z \omega^{(i)}+\sum_{i\in
I_{\text{\RIGHTcircle}}}\Z \omega_i +\sum_{i\in I_{\otimes}} \sum_{j\in
T_\infty}\Z\omega^{(i)}_j,
\end{equation*}
where $\omega^{(i)}$ is the fundamental weight with respect to $d^{(i)}$, and $\omega^{(i)}_j$ is the fundamental
weight corresponding to $\alpha^{(i)}_j$. For $i\in I_{\otimes}$, we
define $\epsilon^{(i)}_j$ and $h^{(i)}_j$ for $j\in T_\infty$ as in
the case of $|I_{\otimes}|=1$ \eqref{epsilon and h}. Hence the
integral weights for $\G$ and $\SG$ are given by
\begin{align*}
&P:=\sum_{i\in I_{\otimes}} \Z \omega^{(i)}+\sum_{i\in I_{\text{\RIGHTcircle}}}\Z\,
\omega_i+\sum_{i\in I_{\otimes}}\sum_{n\in\{-1\}\cup\N}\Z\epsilon^{(i)}_n,
 \\
&\ov{P}:=\sum_{i\in I_{\otimes}}\Z \omega^{(i)}+\sum_{i\in I_{\text{\RIGHTcircle}}}\Z\,
\omega_i+\sum_{i\in I_{\otimes}}\sum_{r\in\{-1\}\cup\left(\hf+\Z_+\right)}\Z\epsilon^{(i)}_r,
\end{align*}
respectively. Now, we assume that the category $\wt{\mc{O}}$ is with
respect to the Levi subalgebra $\wt{\mf l}$ associated to
$\wt{\Pi}\setminus\{\,\alpha_{-1}^{(i)}\,|\,i\in I_{\otimes}\,\}$.
In this case, $P^+$ is the subset of $P$ consisting of
\begin{equation}\label{integral weights}
\la=\sum_{i\in I_{\otimes}}\kappa^{(i)}\omega^{(i)}+\sum_{i\in
I_{\text{\RIGHTcircle}}}\kappa_i\omega_i +\sum_{i\in
I_{\otimes}}\sum_{n\in\{-1\}\cup\N}\la^{(i)}_n\epsilon^{(i)}_n,
\end{equation}
where $\kappa_i\in\Z_+$ (respectively $\kappa_i\in 2\Z_+$) for $i\in
\Icirc$ (respectively $i\in \Ibullet$), and
$(\la_1^{(i)},\la^{(i)}_2,\ldots)\in \mc{P}$ for $i\in I_{\otimes}$.
Then $\ov{P}^+$ and $\wt{P}^+$ are given by the images of $P^+$ under
$\natural$ and $\theta$, respectively, as in the case of
$|I_{\otimes}|=1$ \eqref{natural and theta}.

%Assume that $\wt{\mc O}$, $\ov{\mc O}$ and $\mc O$ are the parabolic categories
%associated with $\wt{\mf l}$, $\ov{\mf l}$ and ${\mf l}$, respectively.
Let ${\mc O}^{\rm int}$ be the full subcategory of $\mc O$
consisting of integrable $\G$-modules and let ${P}^{++}_\G$ be the
subset of $P^+$ consisting of weights of the form in \eqref{integral
weights} with $(\la_{-1}^{(i)},\la_1^{(i)},\la^{(i)}_2,\ldots)\in
\mc{P}$ for $i\in I_{\otimes}$. Note that ${\mc O}^{\rm int}$ is a
semisimple tensor category, whose irreducible objects are $L(\la)$
for $\la\in {P}^{++}_\G$ by \cite{K2}.

Let $\ov{\mc O}^{\rm int}$ be the full subcategory of $\ov{\mc O}$
consisting of $\SG$-modules $\ov{M}$ such that
\begin{equation}\label{polynomial weight}
{\rm wt}\left(\ov{M}\right)\subseteq \sum_{i\in I_{\otimes}}\Z
\omega^{(i)}+ \sum_{i\in I_{\text{\RIGHTcircle}}}\Z\,\omega_i +\sum_{i\in
I_{\otimes}}\sum_{r\in\{-1\}\cup\left(\hf+\Z_+\right)}\Z_+\epsilon^{(i)}_r.
\end{equation}

\begin{lem}\label{lem:integrable}
$\ov{\mc O}^{\rm int}$ is a semisimple tensor category equivalent to
${\mc O}^{\rm int}$ under super duality, whose irreducible objects
are $\ov{L}(\la^\natural)$ for $\la\in {P}^{++}_\G$.
\end{lem}

\begin{proof}
Given $i\in I_{\otimes}$, let $\wt{\mf{gl}}^{(i)}$ be the subalgebra
of $\DG$ isomorphic to $\gl(\infty|\infty)$, which  corresponds to
$\{\,\alpha^{(i)}_j\,|\,j\in T_\infty\,\}$ with the  Cartan
subalgebra spanned by $h^{(i)}_j$  for $j\in T_\infty$ (see Remark
\ref{rem:gl}). We put $\ov{\gl}^{(i)}=\SG\cap
\wt{\gl}^{(i)}\cong\gl(1|\infty)$ and $\gl^{(i)}=\G\cap
\wt{\gl}^{(i)}\cong \gl(\infty)$.

Let $\ov{M}\in \ov{\mc O}^{\rm int}$ be given. By \eqref{polynomial
weight}, $\ov{M}$ is a polynomial $\ov{\gl}^{(i)}$-module
for $i\in I_{\otimes}$, and hence completely reducible over
$\ov{\gl}^{(i)}$ \cite{CK}. Let $M$ be the $\G$-module in $\mc O$
corresponding to $\ov{M}$ under super duality. Since the super
duality between $\ov{\mc O}$ and $\mc O$ naturally induces a super
duality between corresponding categories of  $\ov{\gl}^{(i)}$ and $\gl^{(i)}$-modules for
each $i\in I_{\otimes}$, $M$ is also a polynomial $\gl^{(i)}$-module
and hence completely reducible over $\gl^{(i)}$. By our choice of $\mf l$, the
$\mf{sl}(2)$ or $\mf{osp}(1|2)$-copy associated with each simple
root of $\G$ acts locally nilpotently on $M$, which implies that
$M\in {\mc O}^{\rm int}$. Hence, $M\cong\bigoplus_{\la\in
{P}^{++}_\G}L(\la)^{\oplus m_\la}$ for some $m_\la\in \Z_+$ and
$\ov{M}\cong\bigoplus_{\la\in
{P}^{++}_\G}\ov{L}(\la^\natural)^{\oplus m_\la}$.

On the other hand, for $\la\in {P}^{++}_\G$, we see that $L(\la)$ is
a polynomial $\gl^{(i)}$-module for $i\in I_{\otimes}$. Again by
using the fact that the super duality between $\ov{\mc O}$ and $\mc
O$ induces an equivalence between corresponding categories of  $\ov{\gl}^{(i)}$ and $\gl^{(i)}$-modules for each $i\in I_{\otimes}$, we conclude that
$\ov{L}(\la^\natural)\in \ov{\mc O}^{\rm int}$. It is clear that
$\ov{\mc O}^{\rm int}$ is closed under tensor product. This
completes the proof.
\end{proof}

Let
\begin{equation*}
P_{\mc G}:=\sum_{i\in I_{\text{\RIGHTcircle}}}\Z\,\omega_i+\sum_{i\in
I_{\otimes}}\left(\Z\epsilon^{(i)}_{-1}+\Z\epsilon^{(i)}_{\hf}\right)
\end{equation*}
be the set of integral weights for ${\mc G}$. For $i\in I$, let
\begin{equation*}
{\mc G}_{(i)}:=
\begin{cases}
\left\langle e_i, f_i, \alpha^\vee_i \right\rangle\cong \mf{sl}(2),
& \text{if $i \in \Icirc$}\,,
 \\[.5em]
\left\langle e_i, f_i, \alpha^\vee_i \right\rangle\cong
\mf{osp}(1|2), & \text{if $i \in \Ibullet$}\,,
 \\[.5em]
\left\langle e_i, f_i, h^{(i)}_{-1}, h^{(i)}_{1/2}
\right\rangle\cong \mf{gl}(1|1), & \text{if $i\in I_{\otimes}$}\,.
\end{cases}
\end{equation*}
We define  ${\mc O}_{\mc G}^{\rm int}$ to be the full subcategory of
${\mc O}_{\mc G}$ consisting of $\mc G$-modules $M$ such that
\begin{itemize}
\item[(1)] $M=\bigoplus_{\mu\in P_{\mc G}}M_\mu$ with ${\rm dim}M_\mu< \infty$,

\item[(2)] $M$ is a locally finite ${\mc G}_{(i)}$-module  for $i\in I_{\text{\RIGHTcircle}}$,

\item[(3)]
$M$ is a polynomial ${\mc G}_{(i)}$-module for $i\in I_{\otimes}$, or  equivalently,

\begin{equation*}
{\rm wt}(M)\subset \sum_{i\in I_{\text{\RIGHTcircle}}}\Z\,\omega_i+\sum_{i\in
I_{\otimes}}\left(\Z_+\epsilon^{(i)}_{-1}+\Z_+\epsilon^{(i)}_{\hf}\right).
\end{equation*}
\end{itemize}
%Here $\alpha_i^\vee:=h^{(i)}_{-1}+ h^{(i)}_{1/2}$ is the coroot
%of $\alpha_{-1}^{(i)}$ for $i\in I_{\otimes}$.
Let ${P}^{++}_{\mc G}$ be the set of weights
\begin{equation}
\sum_{i\in I_{\text{\RIGHTcircle}}}\kappa_i\omega_i +\sum_{i\in I_{\otimes}}
\left(\la^{(i)}_{-1}\epsilon^{(i)}_{-1}+\la^{(i)}_{1}\epsilon^{(i)}_{\hf}\right)
\end{equation}
such that
\begin{itemize}
\item[(1)] $\kappa_i\in\Z_+$ (respectively $\kappa_i\in 2\Z_+$) for
$i\in \Icirc$ (respectively $i\in \Ibullet$),

\item[(2)] $\la^{(i)}_{-1},
\la^{(i)}_{1}\in\Z_+$, where $\la^{(i)}_{1}\neq 0$ only if
$\la^{(i)}_{-1}\neq 0$.
\end{itemize}

\begin{thm}\label{thm:semisimple}
The category ${\mc O}_{\mc G}^{\rm int}$ is a semisimple tensor category, whose
irreducible objects are highest weight $\mc G$-modules with highest
weight $\Lambda\in {P}^{++}_{\mc G}$. Furthermore, the
Littlewood-Richardson rule in ${\mc O}_{\mc G}^{\rm int}$ is
determined by that of a symmetrizable anisotropic Kac-Moody Lie
superalgebra.
\end{thm}

\begin{proof}
For $\Lambda\in {P}_{\mc G}$, let $L_{\mc G}(\Lambda)$ be the
irreducible  $\mc{G}$-module in $\mc{O}_{\mc G}$ with highest weight
$\Lambda$. Suppose that $L_{\mc G}(\Lambda)\in\mc{O}_{\mc G}^{\rm
int}$. Then by definition of $\mc{O}_{\mc G}^{\rm int}$, we have
$\Lambda\in {P}^{++}_{\mc G}$. Conversely, if $\Lambda\in
{P}^{++}_{\mc G}$, then there exists  a unique $\la\in {P}^{++}_\G$
such that $\Lambda=\la^\natural$ and $\langle \lambda^\natural,
h^{(i)}_j \rangle =0$ for $i\in I_{\otimes}$ and $j\gg 0$, that is,
the coefficients of $\omega^{(i)}$ in $\la^\natural$ are $0$. Recall
that $\omega^{(i)}=-\epsilon^{(i)}_{-1}+\epsilon^{(i)}_{\hf}$ as a
weight for $\mc{G}$. We have $\ov{L}(\la^\natural)\in \ov{\mc
O}^{\rm int}$ by Lemma \ref{lem:integrable}, and hence $L_{\mc
G}(\Lambda)=\mf{tr}^\infty_{-1}(\ov{L}(\la^\natural))\in \mc{O}_{\mc
G}^{\rm int}$ by \eqref{polynomial weight}. Therefore $\{\,L_{\mc
G}(\Lambda)\,|\,\Lambda\in {P}^{++}_{\mc G}\,\}$ is a complete list
of irreducible modules in $\mc{O}_{\mc G}^{\rm int}$.

Let $M_{\mc G}$ be a highest weight $\mc G$-module in ${\mc O}_{\mc
G}^{\rm int}$ with highest weight  $\Lambda\in {P}^{++}_{\mc G}$,
where $\Lambda=\lambda^\natural$ for some $\la\in {P}^{++}_\G$. We
can check that there exists a highest weight $\SG$-module $\ov{M}\in
\ov{\mc O}$ with highest weight $\la^\natural$ such that
$\mf{tr}^\infty_{-1}(\ov{M})=M_{\mc G}$ by the same argument as in
\cite[Lemma 3.10]{Kw}. Let $M$ be a $\G$-module in $\mc O$
corresponding to $\ov{M}$ via super duality. Then $M$ is also a
highest weight $\G$-module with highest weight $\la$
\cite[Propositions 6.16 and 6.38]{CW}.

Suppose that $M_{\mc G}$ is not irreducible. Then there exists a
proper submodule $N_{\mc G}$ with a non-trivial maximal weight
$\mu^\natural$ for some $\mu\in {P}^{++}_\G$ with $\mu\neq \la$.
Since $\ov{M}$ is $\ov{\mf l}$-semisimple and
$\mf{tr}^{\infty}_{-1}$ maps an irreducible $\ov{\mf{l}}$-module to
an irreducible $(\ov{\mf{l}}\cap \mc{G})$-module or $0$, there
exists an irreducible $\ov{\mf l}$-submodule of $\ov{M}$ with
highest weight $\mu^\natural$. Hence by Lemma \ref{lem:chL}, $M$ has
an irreducible $\mf{l}$-submodule with highest weight $\mu$. In
particular, $\la-\mu\in \sum_{\alpha\in \Pi}\Z_+\alpha$ where $\Pi$
is the set of simple roots of $\G$. On the other hand, $\lambda$ and
$\mu$ are two dominant integral weights appearing in a highest
weight $\G$-module, which implies that $\la=\mu$ by \cite[Corollary
2.6 and Lemma 10.3]{K}. This is a contradiction. Therefore, $M_{\mc
G}=L_{\mc G}(\Lambda)$.

Now, let $\Lambda, \Lambda'\in{P}^{++}_{\mc G}$ be given. Suppose
that we have an exact sequence
\begin{equation}\label{eq:ses}
0 \longrightarrow L_{\mc G}(\Lambda') \longrightarrow M
\longrightarrow L_{\mc G}(\Lambda)\longrightarrow 0
\end{equation}
for some $M\in \mc{O}_{\mc G}^{\rm int}$. Suppose that
$\Lambda-\Lambda'\not\in \sum_{\alpha\in \Pi_{\mc G}}\Z_+\alpha$.
Then $\Lambda'$ is a maximal weight in $M$. Let $f\in
M^\vee_{\Lambda'}$ be such that $\langle f, v \rangle\neq 0$, where
$v\in M_{\Lambda'}$. Since $M^\vee\in \mc{O}_{\mc G}^{\rm int}$ and
$f$ is a maximal weight vector, the $\mc G$-submodule of $M^\vee$
generated by $f$ is isomorphic to $L_{\mc G}(\Lambda')\subset
M^\vee$ by the argument in the previous paragraph. Composing the
embedding $L_{\mc G}(\Lambda') \hookrightarrow M$ with the dual of
$L_{\mc G}(\Lambda')\subset M^\vee$ with respect to $\vee$, we have
a non-zero map
\begin{equation*}
L_{\mc G}(\Lambda') \hookrightarrow M\cong
M^{\vee\vee}\rightarrow L_{\mc G}(\Lambda')^\vee\cong L_{\mc
G}(\Lambda').
\end{equation*}
This implies that \eqref{eq:ses} splits, and hence ${\rm
Ext}^1_{\mc{O}_{\mc G}^{\rm int}}(L_{\mc G}(\Lambda),L_{\mc
G}(\Lambda'))=0$. The same argument also applies to the case when
$\Lambda=\Lambda'$. If $\Lambda-\Lambda'\in \sum_{\alpha\in \Pi_{\mc
G}}\Z_+\alpha$ with $\Lambda\neq \Lambda'$, then we can show that
${\rm Ext}^1_{\mc{O}_{\mc G}^{\rm int}}(L_{\mc G}(\Lambda),L_{\mc
G}(\Lambda'))=0$ by considering the dual of \eqref{eq:ses}.
Therefore, $\mc{O}_{\mc G}^{\rm int}$ is a semisimple tensor
category since it is closed under tensor product.

By Lemma \ref{lem:integrable}, the decomposition of a tensor product
of irreducible modules in $\mc{O}_{\mc G}^{\rm int}$ is determined
by that of ${\mc O}^{\rm int}$. This completes the proof.
\end{proof}

\begin{rem}\label{rem:finite-dimensional case}
In addition to $G(3)$, $F(3|1)$, and $D(2|1,\alpha)$  $(\alpha\in \N)$,  we list the Dynkin diagrams for $\mc{G}$ that correspond to finite-dimensional Lie superalgebras:\vskip 5mm
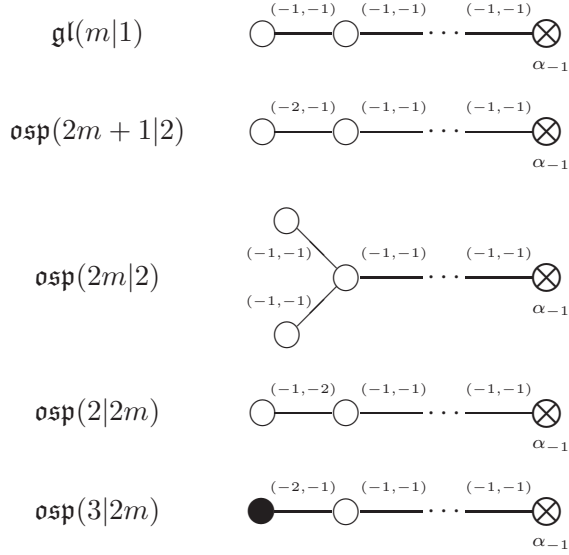
\begin{center}
\hskip -4cm \setlength{\unitlength}{0.25in}
\begin{picture}(12,2)
\put(3,1){\makebox(0,0)[c]{$\gl(m|1)$}}
\put(6.45,1){\makebox(0,0)[c]{$\bigcirc$}}
\put(8.2,1){\makebox(0,0)[c]{$\bigcirc$}}
\put(10.3,1){\makebox(0,0)[c]{$\cdots$}}
\put(12.4,1){\makebox(0,0)[c]{$\bigotimes$}}
\put(6.7,1){\line(1,0){1.2}}%{\Large $\Longrightarrow$}
\put(8.5,1){\line(1,0){1.3}}
\put(10.7,1){\line(1,0){1.4}}
\put(6.6,1.5){\tiny{$_{(-1,-1)}$}}
\put(8.5,1.5){\tiny{$_{(-1,-1)}$}}
\put(10.7,1.5){\tiny{$_{(-1,-1)}$}}
\put(12.5,0.3){\makebox(0,0)[c]{\tiny $\alpha_{-1}$}}
\end{picture}

\hskip -4cm \setlength{\unitlength}{0.25in}
\begin{picture}(12,2)
\put(3,1){\makebox(0,0)[c]{$\osp(2m+1|2)$}}
\put(6.45,1){\makebox(0,0)[c]{$\bigcirc$}}
\put(8.2,1){\makebox(0,0)[c]{$\bigcirc$}}
\put(10.3,1){\makebox(0,0)[c]{$\cdots$}}
\put(12.4,1){\makebox(0,0)[c]{$\bigotimes$}}
\put(6.7,1){\line(1,0){1.2}}%{\Large $\Longrightarrow$}
\put(8.5,1){\line(1,0){1.3}}
\put(10.7,1){\line(1,0){1.4}}
\put(6.6,1.5){\tiny{$_{(-2,-1)}$}}
\put(8.5,1.5){\tiny{$_{(-1,-1)}$}}
\put(10.7,1.5){\tiny{$_{(-1,-1)}$}}
\put(12.5,0.3){\makebox(0,0)[c]{\tiny $\alpha_{-1}$}}
\end{picture}

\hskip -4cm \setlength{\unitlength}{0.25in}
\begin{picture}(12,3)
\put(3,1){\makebox(0,0)[c]{$\osp(2m|2)$}}
\put(6.95,2.2){\makebox(0,0)[c]{$\bigcirc$}}
\put(6.96,-0.2){\makebox(0,0)[c]{$\bigcirc$}}
\put(8.2,1){\makebox(0,0)[c]{$\bigcirc$}}
\put(10.3,1){\makebox(0,0)[c]{$\cdots$}}
\put(12.4,1){\makebox(0,0)[c]{$\bigotimes$}}
\put(8,1.2){\line(-1,1){0.82}}%{\Large $\Longrightarrow$}
\put(8,0.8){\line(-1,-1){0.82}}
\put(8.5,1){\line(1,0){1.3}}
\put(10.7,1){\line(1,0){1.4}}
\put(6.1,1.5){\tiny{$_{(-1,-1)}$}}
\put(6.1,0.5){\tiny{$_{(-1,-1)}$}}
\put(8.5,1.5){\tiny{$_{(-1,-1)}$}}
\put(10.7,1.5){\tiny{$_{(-1,-1)}$}}
\put(12.5,0.3){\makebox(0,0)[c]{\tiny $\alpha_{-1}$}}
\end{picture}\vskip 5mm

\hskip -4cm \setlength{\unitlength}{0.25in}
\begin{picture}(12,2)
\put(3,1){\makebox(0,0)[c]{$\osp(2|2m)$}}
\put(6.45,1){\makebox(0,0)[c]{$\bigcirc$}}
\put(8.2,1){\makebox(0,0)[c]{$\bigcirc$}}
\put(10.3,1){\makebox(0,0)[c]{$\cdots$}}
\put(12.4,1){\makebox(0,0)[c]{$\bigotimes$}}
\put(6.7,1){\line(1,0){1.2}}%{\Large $\Longrightarrow$}
\put(8.5,1){\line(1,0){1.3}}
\put(10.7,1){\line(1,0){1.4}}
\put(6.6,1.5){\tiny{$_{(-1,-2)}$}}
\put(8.5,1.5){\tiny{$_{(-1,-1)}$}}
\put(10.7,1.5){\tiny{$_{(-1,-1)}$}}
\put(12.5,0.3){\makebox(0,0)[c]{\tiny $\alpha_{-1}$}}
\end{picture}

\hskip -4cm \setlength{\unitlength}{0.25in}
\begin{picture}(12,2)
\put(3,1){\makebox(0,0)[c]{$\osp(3|2m)$}}
\put(6.7,1){\makebox(0,0)[c]{\circle*{0.55}}}
\put(8.2,1){\makebox(0,0)[c]{$\bigcirc$}}
\put(10.3,1){\makebox(0,0)[c]{$\cdots$}}
\put(12.4,1){\makebox(0,0)[c]{$\bigotimes$}}
\put(6.7,1){\line(1,0){1.2}}%{\Large $\Longrightarrow$}
\put(8.5,1){\line(1,0){1.3}}
\put(10.7,1){\line(1,0){1.4}}
\put(6.6,1.5){\tiny{$_{(-2,-1)}$}}
\put(8.5,1.5){\tiny{$_{(-1,-1)}$}}
\put(10.7,1.5){\tiny{$_{(-1,-1)}$}}
\put(12.5,0.3){\makebox(0,0)[c]{\tiny $\alpha_{-1}$}}
\end{picture}

\end{center}
where $m\geq 2$. We would like to remark that in the above cases except for $\gl(m|1)$, any non-trivial irreducible $\mc{G}$-module in $\mc{O}_{\mc G}^{\rm int}$ is infinite-dimensional.  Indeed, when $\mc G$ is orthosymplectic, the irreducible ${\mc G}$-modules  in $\mc{O}_{\mc G}^{\rm int}$ are the oscillator modules studied in \cite{CZ, CKW}.
\end{rem}

\begin{rem}
A notion of integrable modules was also introduced in \cite{S}, which is different from our notion of irreducible modules in $\mc{O}^{\rm int}_{\mc{G}}$.
\end{rem}

\begin{rem}\label{rem:integrable for general SG_n}
In general, for $n\in \N\cup\{-1\}$, let $\ov{\mc{O}}^{\rm int}_n$
be a full subcategory of $\SG_n$-modules $\ov{M}$ in $\ov{\mc{O}}_n$
such that
\begin{equation}\label{polynomial weight-finite rank}
{\rm wt}\left(\ov{M}\right)\subseteq   \sum_{i\in
I_{\text{\RIGHTcircle}}}\Z\,\omega_i +\sum_{i\in
I_{\otimes}}\sum_{r\in\{-1,\frac{1}{2},\frac{3}{2}\ldots,
\frac{2n+1}{2}\}}\Z_+\epsilon^{(i)}_r.
\end{equation}
Note that $\ov{\mc{O}}^{\rm int}_{-1}=\mc{O}^{\rm int}_{\mc{G}}$ and
the condition \eqref{polynomial weight-finite rank} is equivalent to
saying that $\ov{M}$ is a polynomial representation of
$\ov{\gl}^{(i)}\cap \SG_n\cong \gl(1|n+1)$ (or $\gl(1|1)$ when
$n=-1$) for $i\in I_{\otimes}$. Also, note that the SGCM for $\SG_n$ does not satisfy \eqref{eq:neq0} for $n\geq 1$ (see \eqref{SG diagram}). Then by the same argument as in
Theorem \ref{thm:semisimple}, we can show that $\ov{\mc{O}}^{\rm
int}_n$ is a semisimple tensor category with irreducible objects
$\ov{L}_n(\la^\natural)$ for $\la\in P_{\G}^{++}$ such that
$\la^\natural\in \ov{P}^+_n$ . It was proved  in \cite[Theorem
3.12]{Kw} for $\SG_n$ being an orthosymplectic Lie superalgebra of finite rank, where $\ov{\mc{O}}^{\rm int}_n$ is defined by a condition
slightly different from but equivalent to \eqref{polynomial
weight-finite rank}.
\end{rem}

\begin{rem}
When $\G$ is a Kac-Moody Lie algebra, classical results on
integrable $\G$-modules (see e.g.~\cite[Theorem 9.1.3]{Ku}) together
with Theorem~\ref{thm:super duality} imply that the irreducible
$\SG_n$-module  $\ov{L}_n(\la^\natural)$ in $\ov{\mc{O}}^{\rm
int}_n$ admits a BGG type resolution in terms of (parabolic) Verma
modules. Furthermore, if we express
$\text{ch}\ov{L}_n(\la^\natural)$ as a linear combination of
characters of parabolic Verma modules, then the nonzero multiplicity
of each such Verma character is $\pm 1$. When $\G$ is an anisotropic Kac-Moody
superalgebra, we may apply \cite{K2, CFLW} to obtain similar results.
\end{rem}

\end{document}